\documentclass[11pt]{amsart}

\usepackage[T1]{fontenc}
\usepackage[utf8]{inputenc}
\usepackage{amsmath,amssymb,amsthm,mathtools}
\usepackage{enumitem}
\usepackage{microtype}
\usepackage{graphicx}
\usepackage[colorlinks=true,linkcolor=blue,citecolor=blue,urlcolor=blue]{hyperref}

\newtheorem{theorem}{Theorem}[section]
\newtheorem{corollary}[theorem]{Corollary}
\newtheorem{lemma}[theorem]{Lemma}
\newtheorem{proposition}[theorem]{Proposition}

\newtheorem{remark}[theorem]{Remark}

\newcommand{\E}{\mathbb E}
\newcommand{\R}{\mathbb R}
\newcommand{\Var}{\operatorname{Var}}
\newcommand{\Inf}{\operatorname{Inf}}
\newcommand{\1}{\mathbf 1}

\newcommand{\ip}[2]{\langle #1,#2\rangle}
\newcommand{\Tcol}{\mathsf T_{\rm col}}
\newcommand{\Ccol}{\mathsf T_{\rm col}}
\newcommand{\Tfrac}{\mathsf T_{\rm frac}}
\newcommand{\Ttheta}{\mathsf T_{\theta}}
\newcommand{\B}{\mathsf B}
\newcommand{\divg}{\operatorname{div}}
\DeclareMathOperator{\dist}{dist}
\newcommand{\dd}{\textnormal{d}}
\newcommand{\A}{A}
\newcommand{\V}{V}
\newcommand{\Stab}{\operatorname{Stab}}

\title[A Beckmann boundary form of Talagrand's conjecture]{A Beckmann boundary form of Talagrand's conjecture on the discrete cube}

\author[P. Ivanisvili]{Paata Ivanisvili}
\address{Department of Mathematics, University of California, Irvine, CA 92697, USA}
\email{pivanisv@uci.edu}

\author[X. Xie]{Xinyuan Xie}
\address{Department of Mathematics, University of California, Irvine, CA 92697, USA}
\email{xinyuax7@uci.edu}

\author[H. Zhang]{Haonan Zhang}
\address{Department of Mathematics, University of South Carolina, Columbia, SC 29208, USA}
\email{haonanzhangmath@gmail.com}

\date{\today}

\begin{document}

\begin{abstract}
We introduce the Beckmann boundary of a Boolean function
\[
   \mathsf{B}(f)=\inf_{\operatorname{div} V=Lf}\mathbb E\|V(x)\|_2.
\]
Here
\[
   L=\sum_iD_i,\qquad  D_i f(x)=\frac{f(x)-f(x^{\oplus i})}{2},
\]
and $\operatorname{div} V(x)=\sum_i  (V_{i}(x)-V_{i}(x^{\oplus i}))$.
This nonlocal quantity is no larger than the usual two-sided, one-sided,
colored, optimized colored, or optimized fractional colored boundaries.
Nevertheless, every nonconstant Boolean $f$ satisfies
\[
   \mathsf{B}(f)\gtrsim \operatorname{Var}(f)
   \sqrt{\log\!\left(1+\frac{1}{\sum_i\operatorname{Inf}_i(f)^2}\right)}.
\]
We also prove strong one-sided fractional spectral estimates.  If
$A\subset\{-1,1\}^n$ and
\[
   h_{A}(x)=\#\{i:x\in A,\ x^{\oplus i}\notin A\},
\]
then, for $0<\alpha<1$,
\[
   \sum_{S\ne\varnothing}|S|^\alpha\widehat{\mathbf 1_{A}}(S)^2
   \lesssim_\alpha \mathbb E\omega_\alpha(h_{A}),
\]
where $\omega_\alpha(m)=\sqrt m$ for $\alpha<1/2$,
$\omega_{1/2}(m)=\sqrt m\log(e+m)$, and
$\omega_\alpha(m)=m^\alpha$ for $\alpha>1/2$.
These profiles are sharp, up to $\alpha$-dependent constants, for majority.
We also show that the comparison is genuinely nonreversible: an explicit
quotient-cube family makes the optimized fractional, and hence optimized
colored, boundary exceed $\mathsf{B}$ by a factor $\gtrsim\sqrt{\log n}$.  We
further obtain a driftless Bernstein-multiplier inequality.
\end{abstract}

\maketitle

\section{Introduction}

Let $\Omega_n=\{-1,1\}^n$ with the uniform probability measure.  For a Boolean function
$f:\Omega_n\to\{-1,1\}$, put $\A:=\{x:f(x)=1\}$.  Its one-sided boundary degree is
\[
   h_{\A}(x):=\#\{i:x\in\A,\ x^{\oplus i}\notin\A\},
\]
where $x^{\oplus i}$ denotes the point obtained by flipping the $i$th coordinate of $x$.  The full, two-sided sensitivity is
\[
   s_f(x)=\#\{i\in[n]:f(x)\ne f(x^{\oplus i})\},
   \qquad [n]:=\{1,\dots,n\}.
\]
Talagrand's classical theorem, in its original one-sided form, gives \cite{Talagrand1993}
\[
   \E\sqrt{h_{\A}(x)}
   \gtrsim \Var(f)\sqrt{\log(e/\Var(f))}.
\]
Talagrand later asked for a strengthening that also captures the KKL theorem, again formulated using the one-sided boundary \cite{KKL,Talagrand1997}.  In the uniform setting and normalization considered here, the conjectured estimate is
\begin{equation}\label{eq:Talagrand-intro}
   \E\sqrt{h_{\A}(x)}
   \ge c\,\Var(f)
   \sqrt{\log\!\left(1+\frac{1}{\sum_{i=1}^n\Inf_i(f)^2}\right)} .
\end{equation}
Here $\Inf_i(f)=\Pr[f(x)\ne f(x^{\oplus i})]$ is the influence of the $i$th variable for $f$. 

Eldan and Gross proved the related two-sided inequality \cite[Theorem~1.4]{EldanGross}
\begin{equation}\label{eq:EG-intro}
   \E\sqrt{s_f(x)}
   \ge c\,\Var(f)
   \sqrt{\log\!\left(1+\frac{1}{\sum_{i=1}^n\Inf_i(f)^2}\right)} .
\end{equation}
Indeed,
\[
   s_f=h_{\A}+h_{\A^c},
   \qquad
   \sqrt{s_f}=\sqrt{h_{\A}}+\sqrt{h_{\A^c}},
\]
because $h_{\A}$ and $h_{\A^c}$ have disjoint supports.  Consequently, Talagrand's one-sided conjecture implies \eqref{eq:EG-intro}, but the converse does not follow.  The two quantities can differ by an unbounded factor: if
$f=2\1_{\{(1,\ldots,1)\}}-1$, then
\[
   \E\sqrt{h_{\A}}=\frac{\sqrt n}{2^n},
   \qquad
   \E\sqrt{s_f}=\frac{n+\sqrt n}{2^n}\geq \sqrt{n}\; \E\sqrt{h_{\A}}.
\]
This distinction is also emphasized in \cite{IvanisviliZhang}.  The stochastic argument in \cite{EldanGross} controls the full squared gradient, or equivalently the two-sided sensitivity, through quadratic variation.  In particular, the one-sided estimate \eqref{eq:Talagrand-intro} does not follow from their published argument as written.  Several alternative proofs of the two-sided inequality \eqref{eq:EG-intro} are now known; see van Handel's argument as recorded in \cite{RosenthalVanHandel}, \cite[Remark~2]{BIM}, and \cite{IvanisviliZhang}.  The last reference also treats biased discrete cubes and Markov spaces with positive Bakry--\'Emery curvature.

Talagrand's one-sided conjecture follows from the robust colored theorem of Eldan--Kindler--Lifshitz--Minzer \cite[Theorems~1.2 and~1.3, and Section~3.4]{EKLM}.  Color every sensitive edge red or blue and write $\chi$ for the coloring.  Define $s_{f,{\rm red}}(x)$ as the number of red sensitive edges adjacent to $x$ when $f(x)=1$, and set $s_{f,{\rm red}}(x)=0$ when $f(x)=-1$.  Similarly, define $s_{f,{\rm blue}}(x)$ as the number of blue sensitive edges adjacent to $x$ when $f(x)=-1$, and set $s_{f,{\rm blue}}(x)=0$ when $f(x)=1$.  The associated colored boundary is
\[
   \Tcol^\chi(f)=
   \E\sqrt{s_{f,{\rm red}}(x)}+
   \E\sqrt{s_{f,{\rm blue}}(x)}.
\]
The all-red coloring gives $\Tcol^\chi(f)=\E\sqrt{h_{\A}}$, while the all-blue coloring gives $\Tcol^\chi(f)=\E\sqrt{h_{\A^c}}$.  We write
\[
   \Ccol(f)=\inf_\chi \Tcol^\chi(f)
\]
for the optimized colored boundary.  Eldan--Kindler--Lifshitz--Minzer proved the Talagrand lower bound for $\Tcol^\chi(f)$ for every coloring $\chi$, and hence also for $\Ccol(f)$.  Taking the all-red coloring yields \eqref{eq:Talagrand-intro}; thus their colored theorem directly settles Talagrand's original one-sided formulation and, in particular, implies the two-sided inequality \eqref{eq:EG-intro}.

The main point of this paper is that the semigroup method proves a still stronger statement.  The colored quantity is edge-local: every sensitive edge is assigned to one of its two endpoints.  We instead allow the boundary charge to flow through the whole cube.  Define a vector field
$\V=(\V_1,\ldots,\V_n)$ on the cube and set
\[
   \divg \V(x)=\sum_{i=1}^n(\V_i(x)-\V_i(x^{\oplus i})).
\]
With
\[
   D_i g(x)=\frac{g(x)-g(x^{\oplus i})}{2},
   \qquad L=\sum_iD_i,
\]
the Beckmann boundary is
\begin{equation}\label{eq:B-intro}
   \B(f)=\inf\left\{
   \E\|\V(x)\|_2:\ \divg \V=Lf
   \right\}.
\end{equation}
The terminology is motivated by the Beckmann formulation of optimal transport \cite{Beckmann}: one minimizes an $L^1$ norm of a vector field subject to a prescribed divergence.  Here the local norm is the Euclidean norm in the coordinate directions of the cube.

This functional is no larger than any colored boundary.  Indeed, a red edge is represented by putting unit field at the $f=1$ endpoint, while a blue edge is represented by putting unit field with the opposite sign at the $f=-1$ endpoint.  It is also no larger than the optimized {\em fractional} colored boundary $\Tfrac(f)$ (see the next section for the definition), where each sensitive edge may be split fractionally between its two endpoints.  More precisely,
\[
   \B(f)\le \Tfrac(f)\le \Ccol(f)\le \Tcol^\chi(f)
\]
for every coloring $\chi$.  Thus a lower bound for $\B(f)$ implies all of the colored, fractional colored, one-sided and two-sided inequalities.

Our main theorem is the following.

\begin{theorem}[Beckmann--Talagrand inequality]\label{thm:main}
There is a universal constant $c>0$ such that for every $n$ and every nonconstant Boolean function $f:\{-1,1\}^n\to\{-1,1\}$,
\begin{equation}\label{eq:main}
   \B(f)
   \ge
   c\,\Var(f)
   \sqrt{\log\!\left(1+\frac{1}{M(f)}\right)},
   \qquad
   M(f)=\sum_{i=1}^n\Inf_i(f)^2 .
\end{equation}
\end{theorem}

The proof is short.  For any admissible $\V$ satisfying $\divg \V=Lf$, the heat semigroup identity gives
\[
   \ip{f}{(I-P_t)f}
   =\int_0^t\ip{Lf}{P_s f}\,\dd s
   =\int_0^t\ip{\divg \V}{P_s f}\,\dd s.
\]
Integration by parts on the cube turns the last inner product into
\[
   2\sum_i\ip{\V_i}{D_iP_s f}.
\]
Cauchy--Schwarz and the standard pointwise reverse Poincar\'e estimate, which follows from the heat-kernel identity \cite[Lemma~2.1]{IVHV},
\[
   \|DP_s f(x)\|_2\le \frac{\|f\|_\infty}{\sqrt{e^{2s}-1}}
\]
then imply
\begin{equation}\label{eq:intro-Beckmann-Poincar\'e}
   \ip{f}{(I-P_t)f}
   \le
   2\arctan\sqrt{e^{2t}-1}\,\E\|\V\|_2 .
\end{equation}
Taking the infimum over $\V$ and setting $t=1/d$ yields
\[
   \B(f)\gtrsim \sqrt d\,W_{\ge d}[f]
\]
with details presented in Section~\ref{sec:Poincar\'e}. The theorem follows by combining this high-degree estimate with the usual hypercontractive low-degree estimate \cite{Bonami,Beckner} and the Keller--Kindler Fourier-tail theorem \cite{KellerKindler}; see also \cite[Theorem~3.4]{EKLM} for the precise form used here.  The proof of the main theorem is completed in Section~\ref{sec:main-proof}. In Section~\ref{sec:Poincar\'e}, we also prove that  $\B(f) \leq \Tfrac(f)\lesssim \B(f)\log(en)$.

The beckman boundary is a genuine strengthening of the colored theorem of Eldan--Kindler--Lifshitz--Minzer \cite[Section~3.4]{EKLM}.  The quantity $\B(f)$ may use nonsensitive edges and global cancellations, whereas colorings and fractional colorings are supported on the sensitivity graph.  In Section \ref{sec:examples} we give explicit examples showing that the reverse comparisons fail and summarize the comparisons.  A singleton shows that a fixed coloring, a fixed fractional coloring, or a one-sided boundary can exceed $\B(f)$ by a factor of $\sqrt n$.  More decisively, an explicit quotient-cube family satisfies
\[
   \frac{\Tfrac(f)}{\B(f)}\gtrsim \sqrt{\log n}
   \qquad\text{and hence}\qquad
   \frac{\Ccol(f)}{\B(f)}\gtrsim \sqrt{\log n},
\]
so the universal comparison from $\B$ to optimized fractional or integral colorings has no dimension-free converse.  This quantitatively confirms that Theorem \ref{thm:main} is strictly stronger than the optimized colored-boundary theorem.

The second part of the paper uses more of the heat-kernel identity of Ivanisvili--van Handel--Volberg \cite[Lemma~2.1]{IVHV}.  The Beckmann--Poincar\'e argument above only uses the rough estimate $\|DP_tf\|_2\lesssim t^{-1/2}$.  If one keeps the exact biased Rademacher variables that appear when differentiating the heat kernel, one obtains strong spectral estimates.  The most transparent endpoint is the following inequality: for every $\A\subset\Omega_n$,
\begin{equation}\label{eq:intro-strong-endpoint}
   \sum_{S\ne\varnothing}\sqrt{|S|}\,\widehat{\1_{\A}}(S)^2
   \lesssim
   \E\big[\sqrt{h_{\A}(x)}\log(e+h_{\A}(x))\big].
\end{equation}
The left-hand side is the strong half-moment of the Fourier spectral degree.  It is stronger than the weak quantity $\sup_d\sqrt d\,W_{\ge d}[\1_{\A}]$: the weak estimate sees the largest Fourier scale, while \eqref{eq:intro-strong-endpoint} sums over all scales.  Equivalently, \eqref{eq:intro-strong-endpoint} is an integrated noise-sensitivity estimate for $\1_{\A}$.

The logarithm in \eqref{eq:intro-strong-endpoint} is a local endpoint correction.  It should be distinguished from the global $L_2/L_1$ logarithms in Talagrand's influence inequality \cite{Talagrand1994}, as revisited in \cite{CorderoEskenazis}.  It is not an artifact of the proof.  For majority, the one-sided boundary is supported on the critical layer; there $h_{\A}\asymp n$ and the layer has measure $\asymp n^{-1/2}$, so the right-hand side of \eqref{eq:intro-strong-endpoint} is $\asymp\log n$.  The standard Fourier asymptotics of majority \cite[Section~5.3]{ODonnell} give $W_{=k}(\mathrm{Maj}_n)\asymp k^{-3/2}$ over the relevant odd levels, and therefore the left-hand side is also $\asymp\log n$.  Thus the endpoint is sharp up to constants.

More generally, we prove a complete first-order spectral moment scale.  For $0<\alpha<1$,
\begin{equation}\label{eq:intro-alpha-scale}
   \sum_{S\ne\varnothing}|S|^\alpha\widehat{\1_{\A}}(S)^2
   \lesssim_\alpha
   \E\omega_\alpha(h_{\A}(x)),
\end{equation}
where
\[
   \omega_\alpha(m)=
   \begin{cases}
      \sqrt m, & 0<\alpha<1/2,\\
      \sqrt m\log(e+m), & \alpha=1/2,\\
      m^\alpha, & 1/2<\alpha<1.
   \end{cases}
\]

\subsection*{Acknowledgments}
P.I. acknowledges partial support from NSF CAREER grant DMS-2152401, NSF grant DMS-2554183, a Simons Fellowship, and a Humboldt Research Fellowship for Experienced Researchers.  H.Z. is supported by NSF DMS-2453408. The authors acknowledge the use of AI tools. All mathematical arguments and proofs in the final manuscript were checked and written by the authors.

\section{Notation and edge-local boundary functionals}\label{sec:notation}

Throughout the paper, $\Omega_n=\{-1,1\}^n$ is equipped with the uniform probability measure, and
\[
   \ip{g}{h}:=\E[g(x)h(x)].
\]
For $x\in\Omega_n$, $x^{\oplus i}$ denotes the point obtained by flipping the $i$th coordinate.  Let
\[
   D_i g(x)=\frac{g(x)-g(x^{\oplus i})}{2},
   \qquad
   Lg=\sum_{i=1}^nD_i g.
\]
Then $D_i$ is a self-adjoint projection on $L^2(\Omega_n)$ and the Walsh character
$\chi_S(x)=\prod_{i\in S}x_i$ satisfies
\[
   D_i\chi_S=\1_{i\in S}\chi_S,
   \qquad
   L\chi_S=|S|\chi_S.
\]
The heat semigroup is $P_t=e^{-tL}$, so
\[
   P_t\chi_S=e^{-t|S|}\chi_S.
\]
For any function $\phi$ on $\{0,1,\ldots,n\}$, the spectral multiplier $\phi(L)$ is defined by
\[
   \phi(L)\chi_S=\phi(|S|)\chi_S.
\]
For $f=\sum_S\widehat f(S)\chi_S$, put
\[
\begin{aligned}
   W_{=d}[f]&=\sum_{|S|=d}\widehat f(S)^2,
   &\qquad
   W_{\ge d}[f]&=\sum_{|S|\ge d}\widehat f(S)^2,\\
   W_{>d}[f]&=\sum_{|S|>d}\widehat f(S)^2,
   &\qquad
   W_{\le d}[f]&=\sum_{1\le |S|\le d}\widehat f(S)^2.
\end{aligned}
\]
Thus $W_{\le d}$ always denotes the nonconstant Fourier mass up to degree $d$, and
\[
   \Var(f)=\sum_{S\ne\varnothing}\widehat f(S)^2.
\]
We write $a\lesssim b$ if $a\le Cb$ for a universal constant $C$, and $a\gtrsim b$ if $b\lesssim a$.

For Boolean $f:\Omega_n\to\{-1,1\}$, define
\[
   \Inf_i(f)=\E|D_i f|=\Pr[f(x)\ne f(x^{\oplus i})].
\]
The equality follows from our normalization: on a sensitive edge, $|D_i f|=1$.

\subsection{Colored and fractional boundaries}
Let $E_f$ be the set of sensitive edges of $f$.  A red--blue coloring $\chi$ of $E_f$ defines
\[
   s_{f,{\rm red}}(x)
   =\#\{i: f(x)=1,\ f(x^{\oplus i})=-1,
       \{x,x^{\oplus i}\}\text{ is red}\},
\]
\[
   s_{f,{\rm blue}}(x)
   =\#\{i: f(x)=-1,\ f(x^{\oplus i})=1,
       \{x,x^{\oplus i}\}\text{ is blue}\}.
\]
Then
\[
   \Tcol^\chi(f)=
   \E\sqrt{s_{f,{\rm red}}(x)}+\E\sqrt{s_{f,{\rm blue}}(x)}.
\]
The all-red coloring is the one-sided boundary of $\A=\{f=1\}$, and the all-blue coloring is the one-sided boundary of $\A^c$.  The optimized colored boundary is
\[
   \Ccol(f)=\inf_\chi \Tcol^\chi(f),
\]
where the infimum ranges over all red--blue colorings of $E_f$.

One may wonder what happens if the coloring assignment were allowed to be fractional. This motivates the following definition. A fractional coloring \(\theta:E_f\to[0,1]\) assigns a number $\theta_e\in[0,1]$ to each sensitive edge $e$.  If $e=\{x,x^{\oplus i}\}$ and $f(x)=1$, $f(x^{\oplus i})=-1$, then $\theta_e$ is the fraction assigned to the positive endpoint and $1-\theta_e$ is the fraction assigned to the negative endpoint.  Define
\[
   \Ttheta(f)=\E_x\left(
      \sum_{\substack{i:\, f(x)=1\\ f(x^{\oplus i})=-1}}
          \theta_{\{x,x^{\oplus i}\}}^2
      +
      \sum_{\substack{i:\, f(x)=-1\\ f(x^{\oplus i})=1}}
          (1-\theta_{\{x,x^{\oplus i}\}})^2
   \right)^{1/2}.
\]
The optimized fractional colored boundary is
\[
   \Tfrac(f)=\inf_{\theta:E_f\to[0,1]}\Ttheta(f).
\]

\begin{proposition}[Fractional and integral colorings are comparable]
\label{prop:fractional-integral-comparable}
For every Boolean function \(f:\Omega_n\to\{-1,1\}\),
\[
   \Tfrac(f) \leq \Ccol(f)\le \sqrt{2}\,\Tfrac(f).
\]
\end{proposition}

\begin{proof}
The first inequality follows directly from the definition, so it remains to prove the second inequality. Fix a fractional coloring \(\theta:E_f\to[0,1]\).  If
\(e=\{x,y\}\in E_f\) with \(f(x)=1\) and \(f(y)=-1\), then \(\theta_e\)
is the fraction assigned to the positive endpoint \(x\), while \(1-\theta_e\)
is the fraction assigned to the negative endpoint \(y\).

We randomly round each sensitive edge independently.  For
\(e=\{x,y\}\in E_f\), with \(f(x)=1\) and \(f(y)=-1\), color \(e\) red with
probability
\[
   p_e
   :=
   \frac{\theta_e^2}{\theta_e^2+(1-\theta_e)^2},
\]
and blue with the complementary probability \(1-p_e\).  Let \(\chi\) denote
the resulting random coloring.

If \(f(x)=1\), then by concavity of the square root,
\[
   \E_\chi\sqrt{s_{f,{\rm red}}(x)}
   \le
   \left(\E_\chi s_{f,{\rm red}}(x)\right)^{1/2}
   =
   \left(
      \sum_{\substack{i:\,f(x)=1\\ f(x^{\oplus i})=-1}}
      p_{\{x,x^{\oplus i}\}}
   \right)^{1/2}.
\]
Since
\[
   \theta_e^2+(1-\theta_e)^2\ge \frac12,
\]
we have
\[
   p_e
   =
   \frac{\theta_e^2}{\theta_e^2+(1-\theta_e)^2}
   \le
   2\theta_e^2.
\]
Therefore
\[
   \E_\chi\sqrt{s_{f,{\rm red}}(x)}
   \le
   \sqrt{2}
   \left(
      \sum_{\substack{i:\,f(x)=1\\ f(x^{\oplus i})=-1}}
      \theta_{\{x,x^{\oplus i}\}}^2
   \right)^{1/2}.
\]

Similarly, if \(f(x)=-1\), then
\[
   1-p_e
   =
   \frac{(1-\theta_e)^2}{\theta_e^2+(1-\theta_e)^2}
   \le
   2(1-\theta_e)^2,
\]
and hence
\[
   \E_\chi\sqrt{s_{f,{\rm blue}}(x)}
   \le
   \sqrt{2}
   \left(
      \sum_{\substack{i:\,f(x)=-1\\ f(x^{\oplus i})=1}}
      \bigl(1-\theta_{\{x,x^{\oplus i}\}}\bigr)^2
   \right)^{1/2}.
\]

Averaging over \(x\), we obtain
\[
   \E_\chi \Tcol^\chi(f)
   \le
   \sqrt{2}\,\Ttheta(f).
\]
Thus 
\[
   \Ccol(f)\le \sqrt{2}\,\Ttheta(f).
\]
Finally, taking the infimum over all fractional colorings \(\theta\) yields
\[
   \Ccol(f)\le \sqrt{2}\,\Tfrac(f).
\]
\end{proof}
\begin{remark}
In addition to being a natural extension of \(\Ccol(f)\), \(\Tfrac(f)\) coincides, up to a normalization factor of 4, with the row--column norm considered by Ben Efraim and Lust-Piquard in \cite{BenEfraimLustPiquard} for Boolean-valued functions.
\end{remark}

\subsection{The Beckmann boundary}
A vector field on the cube is a function
\[
   \V:\Omega_n\to\R^n,
   \qquad
   \V(x)=(\V_1(x),\ldots,\V_n(x)).
\]
We define its divergence by
\begin{equation}\label{eq:div-def}
   \divg \V(x)=\sum_{i=1}^n\big(\V_i(x)-\V_i(x^{\oplus i})\big).
\end{equation}
The Beckmann boundary of $f$ is
\begin{equation}\label{eq:B-def}
   \B(f)=\inf\left\{
        \E\|\V(x)\|_2:
        \divg \V=Lf
   \right\}.
\end{equation}
The feasible set is nonempty.  For example, $\V_i=f/2$ satisfies $\divg \V=Lf$.

{\begin{remark}\label{rem:endpoint-fields}
If $\V_0$ is a vector field on $\{-1,1\}^n$ with $\divg\V_0=Lf$, then
\[
   \B(f)
   =
   \inf_{\divg H=0}\E\|\V_0(x)+H(x)\|_2,
\]
i.e., the Beckmann boundary is a quotient norm: adding a divergence-free field
may change the representative. This freedom allows nonsensitive edges and
cancellations, as exploited in the quotient-cube construction of
Theorem~\ref{thm:unbounded-separation}.
\end{remark}}

\begin{proposition}[Hierarchy]\label{prop:hierarchy}
For every Boolean function $f$,
\begin{equation}\label{eq:hierarchy}
   \B(f)\le \Tfrac(f)\le \Ccol(f)\le \Tcol^\chi(f)
\end{equation}
for every red--blue coloring $\chi$. In particular,
\[
   \B(f)\le \Tcol^\chi(f)
\]
for every red--blue coloring $\chi$, and
\[
   \B(f)\le \E\sqrt{h_{\{f=1\}}},
   \qquad
   \B(f)\le \E\sqrt{h_{\{f=-1\}}}.
\]
\end{proposition}

\begin{proof}
Fix a fractional coloring $\theta$.  We build a field $\V^\theta$.  If $e=\{x,x^{\oplus i}\}$ is sensitive, $f(x)=1$ and $f(x^{\oplus i})=-1$, set
\[
   \V_i^\theta(x)=\theta_e,
   \qquad
   \V_i^\theta(x^{\oplus i})=\theta_e-1.
\]
On nonsensitive edges set the corresponding component equal to $0$.  Then on each sensitive edge
\[
   \V_i^\theta(x)-\V_i^\theta(x^{\oplus i})=1=D_i f(x)
\]
at the positive endpoint, and the same identity with both sides negated holds at the negative endpoint.  On nonsensitive edges both sides vanish.  Hence
\[
   \divg \V^\theta=Lf.
\]
Moreover,
\[
   \E\|\V^\theta(x)\|_2=\Ttheta(f).
\]
Taking the infimum over all feasible vector fields gives $\B(f)\le\Ttheta(f)$, and taking the infimum over $\theta$ gives $\B(f)\le\Tfrac(f)$.  Since integral colorings form a subclass of fractional colorings, $\Tfrac(f)\le\Ccol(f)$, while the definition of $\Ccol$ gives $\Ccol(f)\le\Tcol^\chi(f)$ for every $\chi$.  The one-sided assertions follow from the all-red and all-blue colorings.
\end{proof}

\section{The Beckmann Poincar\'e inequality}\label{sec:Poincar\'e}

We first recall the standard pointwise reverse Poincar\'e estimate, which follows from the heat-kernel representation in \cite[Lemma~2.1]{IVHV}; see also \cite[Section 6]{ILvHV} for a semigroup argument. Write $Dg:=(D_i g)_{i\in[n]}$.

\begin{lemma}[Reverse Poincar\'e]\label{lem:reverse-Poincar\'e}
For every $g:\Omega_n\to\R$, every $t>0$ and every $x\in\Omega_n$,
\begin{equation}\label{eq:reverse-Poincar\'e}
   \|DP_tg(x)\|_2^2
   \le
   \frac{P_tg^2(x)-(P_tg(x))^2}{e^{2t}-1}.
\end{equation}
In particular, if $\|g\|_\infty\le1$, then
\begin{equation}\label{eq:reverse-Poincar\'e-infty}
   \|DP_tg(x)\|_2\le\frac{1}{\sqrt{e^{2t}-1}}.
\end{equation}
\end{lemma}

One proof is recalled in Appendix \ref{app:heat-kernel}.

\begin{theorem}[Beckmann--Poincar\'e inequality]\label{thm:beckmann-Poincar\'e}
Let $f:\Omega_n\to\R$ with $\|f\|_\infty\le1$.  If $\V$ is a vector field satisfying
$\divg \V=Lf$, then for every $t>0$,
\begin{equation}\label{eq:beckmann-Poincar\'e}
   \ip{f}{(I-P_t)f}
   \le
   2\arctan\sqrt{e^{2t}-1}\,\E\|\V(x)\|_2.
\end{equation}
Consequently, for $t>0$,
\begin{equation}\label{eq:beckmann-Poincar\'e-small-t}
   \ip{f}{(I-P_t)f}
   \lesssim
   \sqrt t\,\B(f).
\end{equation}
\end{theorem}

\begin{proof}
Since $P_t=e^{-tL}$,
\[
   \ip{f}{(I-P_t)f}
   =\int_0^t\ip{Lf}{P_s f}\,\dd s.
\]
Using $Lf=\divg \V$ and the definition \eqref{eq:div-def},
\[
   \ip{\divg \V}{P_s f}
   =\sum_{i=1}^n\E\big[(\V_i(x)-\V_i(x^{\oplus i}))P_s f(x)\big].
\]
Changing variables $x\mapsto x^{\oplus i}$ in the second term gives
\[
   \ip{\divg \V}{P_s f}
   =\sum_{i=1}^n\E \V_i(x)\big(P_s f(x)-P_s f(x^{\oplus i})\big)
   =2\sum_{i=1}^n\ip{\V_i}{D_iP_s f}.
\]
Therefore
\[
   \ip{f}{(I-P_t)f}
   =2\int_0^t\E\sum_i \V_i(x)D_iP_s f(x)\,\dd s.
\]
By Cauchy--Schwarz in the coordinate index and Lemma \ref{lem:reverse-Poincar\'e},
\[
  \left| \E\sum_i \V_iD_iP_s f\right|
   \le
   \E\big[\|\V(x)\|_2\|DP_s f(x)\|_2\big]
   \le
   \frac{\E\|\V(x)\|_2}{\sqrt{e^{2s}-1}}.
\]
Thus
\[
   \ip{f}{(I-P_t)f}
   \le
   2\E\|\V\|_2\int_0^t\frac{\dd s}{\sqrt{e^{2s}-1}}.
\]
The integral is explicit:
\[
   \int_0^t\frac{\dd s}{\sqrt{e^{2s}-1}}
   =\arctan\sqrt{e^{2t}-1}.
\]
This gives \eqref{eq:beckmann-Poincar\'e}.  Taking the infimum over all admissible $\V$ gives the same estimate with $\B(f)$ in place of $\E\|\V\|_2$.  Finally,
\[
   \arctan\sqrt{e^{2t}-1}\lesssim \sqrt t
\]
for $t>0$.
\end{proof}

\begin{corollary}[High-degree estimate]\label{cor:high-degree}
For every Boolean $f:\Omega_n\to\{-1,1\}$ and every $d\ge1$,
\begin{equation}\label{eq:high-degree}
   \B(f)\gtrsim \sqrt d\,W_{\ge d}[f].
\end{equation}
\end{corollary}

\begin{proof}
Taking Fourier expansions,
\[
   \ip{f}{(I-P_t)f}
   =\sum_{S\ne\varnothing}(1-e^{-t|S|})\widehat f(S)^2.
\]
With $t=1/d$, the terms with $|S|\ge d$ contribute at least $(1-e^{-1})W_{\ge d}[f]$.  Theorem \ref{thm:beckmann-Poincar\'e} gives
\[
   W_{\ge d}[f]
   \lesssim
   d^{-1/2}\B(f),
\]
which is \eqref{eq:high-degree}.
\end{proof}

\begin{proposition}[Variational formula and strong spectral upper bound]\label{prop:strong-spectral-upper}
For every function $\rho:\Omega_n\to(0,\infty)$ with $\E\rho=1$, define
\[
   L_\rho g(x)
   :=\sum_{i=1}^n
      \frac{D_i g(x)}{\rho(x)+\rho(x^{\oplus i})}.
\]
Then $L_\rho$ is positive and self-adjoint, and every Boolean $f$ satisfies
\begin{equation}\label{eq:Tfrac-variational}
   \Tfrac(f)^2
   =\frac12\inf_{\substack{\rho>0\\ \E\rho=1}}
      \ip{f}{L_\rho f}
   =\frac12\inf_{\substack{\rho>0\\ \E\rho=1}}
      \E\sum_{i=1}^n
      \frac{(D_i f(x))^2}{\rho(x)+\rho(x^{\oplus i})}.
\end{equation}
Consequently,
\begin{equation}\label{eq:strong-spectral-upper}
   \B(f)\le \Tfrac(f)
   \le \frac12\ip{f}{L^{1/2}f}
   =\frac12\sum_{k=1}^n\sqrt{k}\,W_{=k}[f].
\end{equation}
In particular, Corollary \ref{cor:high-degree} and summation by parts give
\begin{equation}\label{eq:weak-strong-sandwich}
   \sup_{1\le d\le n}\sqrt d\,W_{\ge d}[f]
   \lesssim \B(f)
   \le \frac12\sum_{d=1}^n
      (\sqrt d-\sqrt{d-1})W_{\ge d}[f].
\end{equation}
The factor $1/2$ in \eqref{eq:strong-spectral-upper} is sharp.
\end{proposition}

\begin{proof}
Put $N=2^n$.  For every incidence of a sensitive edge $e$ at a vertex
$x$, let $a_{x,e}$ be the fraction of $e$ assigned to $x$.  Thus
$a_{x,e}\in[0,1]$ and $a_{x,e}+a_{y,e}=1$ whenever
$e=\{x,y\}\in E_f$, and
\[
   N\Tfrac(f)
   =\inf_a\sum_{x\in\Omega_n}
      \left(\sum_{\substack{e\in E_f\\e\ni x}}a_{x,e}^2\right)^{1/2}.
\]
Using
\[
   \|u\|_2
   =\inf_{r>0}\frac12\left(r+\frac{\|u\|_2^2}{r}\right)
\]
at each vertex gives
\[
   N\Tfrac(f)
   =\frac12\inf_{\substack{r_x>0\\a_{x,e}+a_{y,e}=1}}
      \left[
         \sum_x r_x+
         \sum_x\sum_{e\ni x}\frac{a_{x,e}^2}{r_x}
      \right].
\]
For fixed $r$, the minimization separates over the sensitive edges.  On
$e=\{x,y\}$,
\[
   \inf_{a+b=1}\left(\frac{a^2}{r_x}+\frac{b^2}{r_y}\right)
   =\frac1{r_x+r_y},
\]
with minimizer
\[
   a=\frac{r_x}{r_x+r_y},
   \qquad
   b=\frac{r_y}{r_x+r_y}.
\]
Hence
\begin{equation}\label{eq:Tfrac-r-variational}
   N\Tfrac(f)
   =\frac12\inf_{r_x>0}
      \left[
         \sum_x r_x+
         \sum_{\{x,y\}\in E_f}\frac1{r_x+r_y}
      \right].
\end{equation}
Since each undirected sensitive edge is counted at both endpoints,
\eqref{eq:Tfrac-r-variational} is equivalently
\begin{equation}\label{eq:Tfrac-r-normalized}
   \Tfrac(f)
   =\inf_{r>0}\left\{
      \frac12\E r+
      \frac14\E\sum_{i=1}^n
      \frac{(D_i f(x))^2}{r(x)+r(x^{\oplus i})}
   \right\}.
\end{equation}
Write $r=c\rho$, where $c>0$ and $\E\rho=1$.  For fixed $\rho$, the
right-hand side of \eqref{eq:Tfrac-r-normalized} becomes
\[
   \frac c2+\frac1{4c}
      \E\sum_{i=1}^n
      \frac{(D_i f(x))^2}{\rho(x)+\rho(x^{\oplus i})}.
\]
Minimizing over $c>0$ proves \eqref{eq:Tfrac-variational}.  Indeed, if
$w_i(x)=(\rho(x)+\rho(x^{\oplus i}))^{-1}$, then
$w_i(x)=w_i(x^{\oplus i})$, so multiplication by $w_i$ commutes with
$D_i$.  It follows that $L_\rho=\sum_iM_{w_i}D_i$ is positive and
self-adjoint and that
\[
   \ip{f}{L_\rho f}
   =\E\sum_iw_i(x)(D_i f(x))^2.
\]

It remains to prove the upper bound in \eqref{eq:strong-spectral-upper}.
The assertion is trivial when $f$ is constant, so assume that $f$ is
nonconstant and put
\[
   H:=\ip{f}{L^{1/2}f}>0,
   \qquad Q_t:=e^{-tL^{1/2}}.
\]
Choose the density
\begin{equation}\label{eq:Poisson-rho}
   \rho(x):=\frac2H\int_0^\infty
      \left|\partial_tQ_tf(x)\right|^2\,\dd t.
\end{equation}
Since $Q_t f\to\E f$ as $t\to\infty$ and $f(x)\ne\E f$ at every
$x$, this density is strictly positive.  Moreover,
\[
\begin{aligned}
   \E\rho
   =\frac2H\int_0^\infty
      \|\partial_tQ_tf\|_2^2\,\dd t
   =\frac2H\int_0^\infty
      \ip{Q_tf}{LQ_tf}\,\dd t
   =1,
\end{aligned}
\]
because the spectral decomposition of $L$ gives
\begin{equation}\label{eq:Poisson-energy}
   \int_0^\infty\ip{Q_tf}{LQ_tf}\,\dd t
   =\frac12\ip{f}{L^{1/2}f}
   =\frac H2.
\end{equation}

Let $y=x^{\oplus i}$ be a sensitive neighbor of $x$, and set
\[
   h(t):=Q_tf(x)-Q_tf(y).
\]
Then $h(0)^2=4$ and $h(t)\to0$.  Integration by parts followed by
Cauchy--Schwarz yields
\[
   4=-\int_0^\infty (h(t)^2)'\,\dd t
   \le2\left(\int_0^\infty h(t)^2\,\dd t\right)^{1/2}
          \left(\int_0^\infty h'(t)^2\,\dd t\right)^{1/2}.
\]
On the other hand,
\[
\begin{aligned}
   \int_0^\infty h'(t)^2\,\dd t
   \le2\int_0^\infty\left(
      |\partial_tQ_tf(x)|^2+|\partial_tQ_tf(y)|^2
   \right)\,\dd t
   =H\bigl(\rho(x)+\rho(y)\bigr).
\end{aligned}
\]
Writing $A=\int_0^\infty h(t)^2\,\dd t$ and
$B=\int_0^\infty h'(t)^2\,\dd t$, the preceding two estimates give
$AB\ge4$ and $B\le H(\rho(x)+\rho(y))$.  Therefore
\begin{equation}\label{eq:rho-edge-bound}
   \frac1{\rho(x)+\rho(y)}
   \le\frac H4\int_0^\infty
      \bigl(Q_tf(x)-Q_tf(y)\bigr)^2\,\dd t
\end{equation}
for every sensitive edge $\{x,y\}$.

Insert \eqref{eq:Poisson-rho} into \eqref{eq:Tfrac-variational}.  Since
$(D_i f)^2$ is the indicator that the $i$th edge at $x$ is sensitive,
\eqref{eq:rho-edge-bound} and $(D_i f)^2\le1$ give
\[
\begin{aligned}
   2\Tfrac(f)^2
   &\le\E\sum_{i=1}^n
      \frac{(D_i f(x))^2}{\rho(x)+\rho(x^{\oplus i})}\\
   &\le\frac H4\int_0^\infty
      \E\sum_{i=1}^n
      \bigl(Q_tf(x)-Q_tf(x^{\oplus i})\bigr)^2\,\dd t\\
   &=H\int_0^\infty\ip{Q_tf}{LQ_tf}\,\dd t
   =\frac{H^2}{2},
\end{aligned}
\]
where the last equality is \eqref{eq:Poisson-energy}.  Thus
$\Tfrac(f)\le H/2$.  Proposition \ref{prop:hierarchy} gives
$\B(f)\le\Tfrac(f)$, and the Fourier identity in
\eqref{eq:strong-spectral-upper} follows from the spectral decomposition
of $L$.  Summation by parts gives \eqref{eq:weak-strong-sandwich}.

Finally, let $f=\chi_S$ with $|S|=d\ge1$.  The choice $\rho\equiv1$ in
\eqref{eq:Tfrac-variational} gives $\Tfrac(f)\le\sqrt d/2$.  Conversely,
for every admissible $\rho$, Jensen's inequality gives
\[
   \E\frac1{\rho(x)+\rho(x^{\oplus i})}
   \ge\frac1{\E\rho(x)+\E\rho(x^{\oplus i})}
   =\frac12
\]
for each $i\in S$.  Hence \eqref{eq:Tfrac-variational} gives
$\Tfrac(f)\ge\sqrt d/2$, proving sharpness.
\end{proof}

\begin{corollary}[Logarithmic reverse Beckmann--fractional comparison]
\label{log-reverse-beckmann-fractional}
For every Boolean function \(f:\Omega_n\to\{-1,1\}\),
\[
   \Tfrac(f)\lesssim \B(f)\log(en).
\]
\end{corollary}

\begin{proof}
By \eqref{eq:strong-spectral-upper} and the summation-by-parts identity
used in \eqref{eq:weak-strong-sandwich},
\[
   \Tfrac(f)
   \le
   \frac12\sum_{d=1}^n
   \bigl(\sqrt d-\sqrt{d-1}\bigr)W_{\ge d}[f].
\]
On the other hand, the left inequality in \eqref{eq:weak-strong-sandwich}
implies, for every \(1\le d\le n\),
\[
   W_{\ge d}[f]\lesssim \frac{\B(f)}{\sqrt d}.
\]
Therefore
\[
   \Tfrac(f)
   \lesssim
   \B(f)
   \sum_{d=1}^n
   \frac{\sqrt d-\sqrt{d-1}}{\sqrt d}.
\]
Finally,
\[
   \frac{\sqrt d-\sqrt{d-1}}{\sqrt d}
   =
   \frac{1}{\sqrt d(\sqrt d+\sqrt{d-1})}
   \le \frac1d,
\]
and hence
\[
   \sum_{d=1}^n
   \frac{\sqrt d-\sqrt{d-1}}{\sqrt d}
   \le
   \sum_{d=1}^n\frac1d
   \le
   \log(en).
\]
The claim follows.
\end{proof}

\section{Proof of the Beckmann--Talagrand inequality}\label{sec:main-proof}

We need two standard Fourier ingredients.  The first is a standard consequence of Bonami--Beckner hypercontractivity \cite{Bonami,Beckner}, equivalently of the usual small-set/level-$d$
Fourier estimates; see   \cite[Chapter 9.5]{ODonnell}.

\begin{lemma}[Low-degree mass from hypercontractivity]\label{lem:low-degree-var}
There are universal constants $c_0,c_1,c_2>0$ such that for every nonconstant Boolean $f:\Omega_n\to\{-1,1\}$ there is an integer
\[
   c_0\log(e/\Var(f)) \le d\le c_1 \log(e/\Var(f))
\]
for which
\[
   W_{\ge d}[f]\ge c_2\Var(f).
\]
Consequently,
\begin{equation}\label{eq:B-classical}
   \B(f)\gtrsim \Var(f)\sqrt{\log(e/\Var(f))}.
\end{equation}
\end{lemma}

The second ingredient is the Keller--Kindler quantitative Fourier-tail theorem \cite{KellerKindler}, in the form recorded in \cite[Theorem~3.4]{EKLM}.

\begin{theorem}[Keller--Kindler Fourier tail]\label{thm:KK}
There exist constants $a_1,a_2>0$ such that for every Boolean function $f$ with
$M(f)=\sum_i\Inf_i(f)^2$ sufficiently small,
\begin{equation}\label{eq:KK-tail}
   W_{\le a_1\log(1/M(f))}[f]
   \le
   M(f)^{a_2}.
\end{equation}
\end{theorem}

The cited formulation is commonly stated for a function $g:\Omega_n\to\{0,1\}$ using the flip-probability influence
\[
   I_i(g):=\Pr[g(x)\ne g(x^{\oplus i})].
\]
To pass to our convention, set $g=(1+f)/2$.  Then
\[
   I_i(g)=\Inf_i(f).
\]
Notice that this is not the same as $\E|D_i g|$: with our difference operator, $D_i g=D_i f/2$ and hence $\E|D_i g|=\Inf_i(f)/2$.  For every nonempty $S$,
\[
   \widehat g(S)=\frac12\widehat f(S),
\]
so
\[
   W_{\le d}[g]=\frac14W_{\le d}[f].
\]
Thus the cited theorem first gives $W_{\le d}[f]\le4M(f)^c$ for some universal $c>0$.  Replacing $c$ by $a_2=c/2$ and decreasing the smallness threshold for $M$ so that $4M^c\le M^{c/2}$ yields Theorem \ref{thm:KK} exactly as stated.

We now prove Theorem \ref{thm:main}.

\begin{proof}[Proof of Theorem \ref{thm:main}]
Let $v=\Var(f)$ and $M=M(f)$.  Fix the smallness threshold in Theorem \ref{thm:KK} sufficiently small that, whenever it applies,
\[
   d=\lfloor a_1\log(1/M)\rfloor\ge1
   \qquad\text{and}\qquad
   d\asymp \log(1/M).
\]
If $M$ is bounded below by that absolute threshold, then
\[
   v\sqrt{\log(1+1/M)}\lesssim v,
\]
and Lemma \ref{lem:low-degree-var} gives $\B(f)\gtrsim v$.  Thus we may assume that $M$ is small enough for Theorem \ref{thm:KK} and, in particular, $M\le1/2$.

There are two cases.  First suppose
\[
   M^{a_2}>\frac12v.
\]
Then $M>(v/2)^{1/a_2}$, and therefore
\[
   \frac1M<\left(\frac2v\right)^{1/a_2}.
\]
It follows that
\[
\begin{aligned}
   \log(1+1/M)
   \le \log\left(1+(2/v)^{1/a_2}\right)
   \le C_{a_2}\log(e/v).
\end{aligned}
\]
Lemma \ref{lem:low-degree-var} consequently gives
\[
   \B(f)
   \gtrsim
   v\sqrt{\log(e/v)}
   \gtrsim
   v\sqrt{\log(1+1/M)}.
\]

Now suppose
\[
   M^{a_2}\le \frac12v.
\]
Set
\[
   d=\lfloor a_1\log(1/M)\rfloor.
\]
By Theorem \ref{thm:KK},
\[
   W_{\le d}[f]\le\frac12v.
\]
Since the total nonconstant Fourier mass is $v$,
\[
   W_{>d}[f]\ge\frac12v.
\]
Because $W_{>d}[f]=W_{\ge d+1}[f]$, Corollary \ref{cor:high-degree} yields
\[
   \B(f)
   \gtrsim
   \sqrt{d+1}\,W_{>d}[f]
   \gtrsim
   v\sqrt{\log(1/M)}.
\]
Finally, when $0<M\le1/2$,
\[
   \log(1/M)\le\log(1+1/M)\le\log(2/M)\le2\log(1/M),
\]
so $\log(1+1/M)\asymp\log(1/M)$.  This gives \eqref{eq:main} in the second case and completes the proof.
\end{proof}

\begin{corollary}[Colored, fractional and one-sided consequences]\label{cor:consequences}
For every nonconstant Boolean $f:\Omega_n\to\{-1,1\}$,
\[
   \Tfrac(f)
   \ge
   c\,\Var(f)
   \sqrt{\log\!\left(1+\frac{1}{M(f)}\right)}.
\]
The same lower bound holds for $\Ccol(f)$ and for $\Tcol^\chi(f)$ for every red--blue coloring $\chi$.  In particular, if $\A=\{f=1\}$, then
\[
   \E\sqrt{h_{\A}}
   \ge
   c\,\Var(f)
   \sqrt{\log\!\left(1+\frac{1}{M(f)}\right)}.
\]
\end{corollary}

\begin{proof}
This is immediate from Theorem \ref{thm:main} and Proposition \ref{prop:hierarchy}.
\end{proof}

\section{Duality and comparison examples}\label{sec:examples}

This section records concrete examples showing that $\B(f)$ is a genuinely smaller object than edge-local colorings.

\subsection{Dual formulation}
The following dual form is useful for exact computations.

\begin{proposition}[Dual Beckmann formulation]\label{prop:dual}
For every Boolean $f$,
\begin{equation}\label{eq:dual}
   \B(f)=
   \sup\left\{
      \E\phi(x)Lf(x):
      2\|D\phi(x)\|_2\le1\ \text{for all }x
   \right\}.
\end{equation}
\end{proposition}

The proof is recalled in Appendix \ref{app:duality}.

\subsection{A fixed coloring or one-sided boundary can be larger by \texorpdfstring{$\sqrt n$}{sqrt(n)}}
Let $o=(-1,\ldots,-1)$ and define
\[
   f_n(o)=-1,
   \qquad
   f_n(x)=1\quad\text{for }x\ne o.
\]
Then the sensitivity graph consists of the $n$ edges adjacent to $o$.

\begin{proposition}\label{prop:singleton}
For the above function,
\[
   \B(f_n)=\frac{\sqrt n}{2^n}.
\]
On the other hand, for the one-sided boundary of $\{f_n=1\}$, equivalently for the all-red coloring,
\[
   \E\sqrt{h_{\{f_n=1\}}}=\frac{n}{2^n}.
\]
Thus
\[
   \frac{\E\sqrt{h_{\{f_n=1\}}}}{\B(f_n)}=\sqrt n.
\]
The same example also gives a fixed fractional coloring $\theta\equiv1$ for which
$\Ttheta(f_n)/\B(f_n)=\sqrt n$.
\end{proposition}

\begin{proof}
For the upper bound on $\B(f_n)$, define $\V_i(o)=-1$ for every $i$ and set all other components equal to $0$.  Then $\divg \V=Lf_n$.  Indeed, $Lf_n(o)=-n$, each neighbor of $o$ has $Lf_n=1$, and all other vertices have $Lf_n=0$.  The cost is
\[
   \E\|\V\|_2=\frac{\sqrt n}{2^n}.
\]

For the matching lower bound, use the dual formulation.  Let
\[
   \phi(x)=\frac{\dist(x,o)}{\sqrt n}.
\]
Flipping one coordinate changes $\phi$ by $1/\sqrt n$, hence
\[
   2\|D\phi(x)\|_2\le1
\]
for every $x$.  Therefore $\phi$ is dual-admissible.  Since $Lf_n(o)=-n$, $Lf_n=1$ on the $n$ neighbors of $o$, and $\phi(o)=0$, $\phi=1/\sqrt n$ on those neighbors,
\[
   \E\phi Lf_n=\frac{\sqrt n}{2^n}.
\]
This proves $\B(f_n)=\sqrt n/2^n$.

The one-sided boundary of $\{f_n=1\}$ consists of the $n$ neighbors of $o$, each with one outgoing edge.  Hence
\[
   \E\sqrt{h_{\{f_n=1\}}}=\frac{n}{2^n}.
\]
The all-red coloring and the fixed fractional coloring $\theta\equiv1$ give the same cost.
\end{proof}

\begin{remark}
Replacing $f_n$ by $-f_n$ reverses the roles of the two sides of the cut.  Thus the two one-sided boundaries are not comparable to each other by dimension-free constants.  The Beckmann boundary is below both.
\end{remark}

\subsection{A fractional packing lower bound}
The next elementary dual bound for fractional endpoint assignments will be used in the quotient-cube construction.

\begin{lemma}[Fractional packing]\label{lem:fractional-packing}
Suppose nonnegative numbers $(\lambda_e)_{e\in E_f}$ satisfy
\[
   \sum_{e\ni x}\lambda_e^2\le1
   \qquad\text{for every }x\in\Omega_n.
\]
Then
\[
   \Tfrac(f)\ge 2^{-n}\sum_{e\in E_f}\lambda_e.
\]
\end{lemma}

\begin{proof}
For a fractional coloring, write $a_{x,e}\ge0$ for the weight assigned to endpoint $x$ of a sensitive edge $e=\{x,y\}$, so $a_{x,e}+a_{y,e}=1$.  Then
\[
\begin{aligned}
   \sum_{e\in E_f}\lambda_e
   &=\sum_{x\in\Omega_n}\sum_{e\ni x}\lambda_ea_{x,e}\\
   &\le\sum_{x\in\Omega_n}
      \left(\sum_{e\ni x}\lambda_e^2\right)^{1/2}
      \left(\sum_{e\ni x}a_{x,e}^2\right)^{1/2}\\
   &\le\sum_{x\in\Omega_n}
      \left(\sum_{e\ni x}a_{x,e}^2\right)^{1/2}.
\end{aligned}
\]
Divide by $2^n$ and take the infimum over all fractional colorings.
\end{proof}

\subsection{An unbounded separation from optimized colorings}
We now give a sequence for which even the optimized colored boundary is larger than the Beckmann boundary by an unbounded factor.

\begin{theorem}[Unbounded reverse separation]\label{thm:unbounded-separation}
For every $k\ge1$ there is a Boolean function $f_k$ on a cube of dimension
\[
   n_k=\frac{2^{k+2}(4^k-1)}{3}
\]
such that
\begin{equation}\label{eq:unbounded-ratio}
   \frac{\Tfrac(f_k)}{\B(f_k)}
   \ge \frac{3}{4\sqrt6}\sqrt{k}.
\end{equation}
Consequently,
\[
   \frac{\Ccol(f_k)}{\B(f_k)}
   \ge \frac{3}{4\sqrt6}\sqrt{k}
   \gtrsim \sqrt{\log n_k}\longrightarrow\infty.
\]
In particular, there is no dimension-free constant $C$ such that
$\Ccol(f)\le C\B(f)$ for every Boolean function $f$.
\end{theorem}

\begin{proof}
We identify a discrete cube with a vector space over $\mathbb F_2$.  Let
\[
   G_k=\mathbb F_2^k.
\]
For $z\ne0$, define
\[
   \ell(z)=\min\{r:z_r=1\},
   \qquad \ell(0)=k+1,
\]
and put
\[
   E_r=\{z\in G_k:\ell(z)=r\},
   \qquad m_r=|E_r|.
\]
Thus
\[
   m_r=2^{k-r}\quad(1\le r\le k),
   \qquad m_{k+1}=1.
\]
Define the alternating-level function
\[
   g_k(z)=(-1)^{\ell(z)}.
\]
For $1\le r\le k$, let
\[
   w_r=8^r=2^{3r}.
\]
For every $a\in G_k\setminus\{0\}$, introduce $w_{\ell(a)}$ cube coordinates carrying the label $a$.  More precisely, set
\[
   I_k=\{(a,s):a\ne0,\ 1\le s\le w_{\ell(a)}\}
\]
and define the surjective linear map
\[
   \pi_k:\mathbb F_2^{I_k}\longrightarrow G_k,
   \qquad
   \pi_k(x)=\sum_{(a,s)\in I_k}x_{a,s}a.
\]
Surjectivity follows because every standard basis vector of $G_k$ occurs as a label.  Finally, set
\[
   f_k=g_k\circ\pi_k.
\]
The dimension of its cube is
\begin{equation}\label{eq:nk-dimension}
   n_k=|I_k|
   =\sum_{r=1}^k m_rw_r
   =\sum_{r=1}^k2^{k-r}8^r
   =\frac{2^{k+2}(4^k-1)}3.
\end{equation}

Figure \ref{fig:quotient-cube-four-layers} records the local geometry of this lift.  If
$\pi_k(x)=y\in E_r$ and $z\in E_{r+1}$, then the unique quotient label joining
them is $a=y+z\in E_r$.  For every $1\le s\le w_r$, the coordinate flip
$x'=x+e_{(a,s)}$ is a distinct cube neighbor of $x$ satisfying
$\pi_k(x')=z$.

\begin{figure}[htbp]
   \centering
   \includegraphics[width=0.98\textwidth]{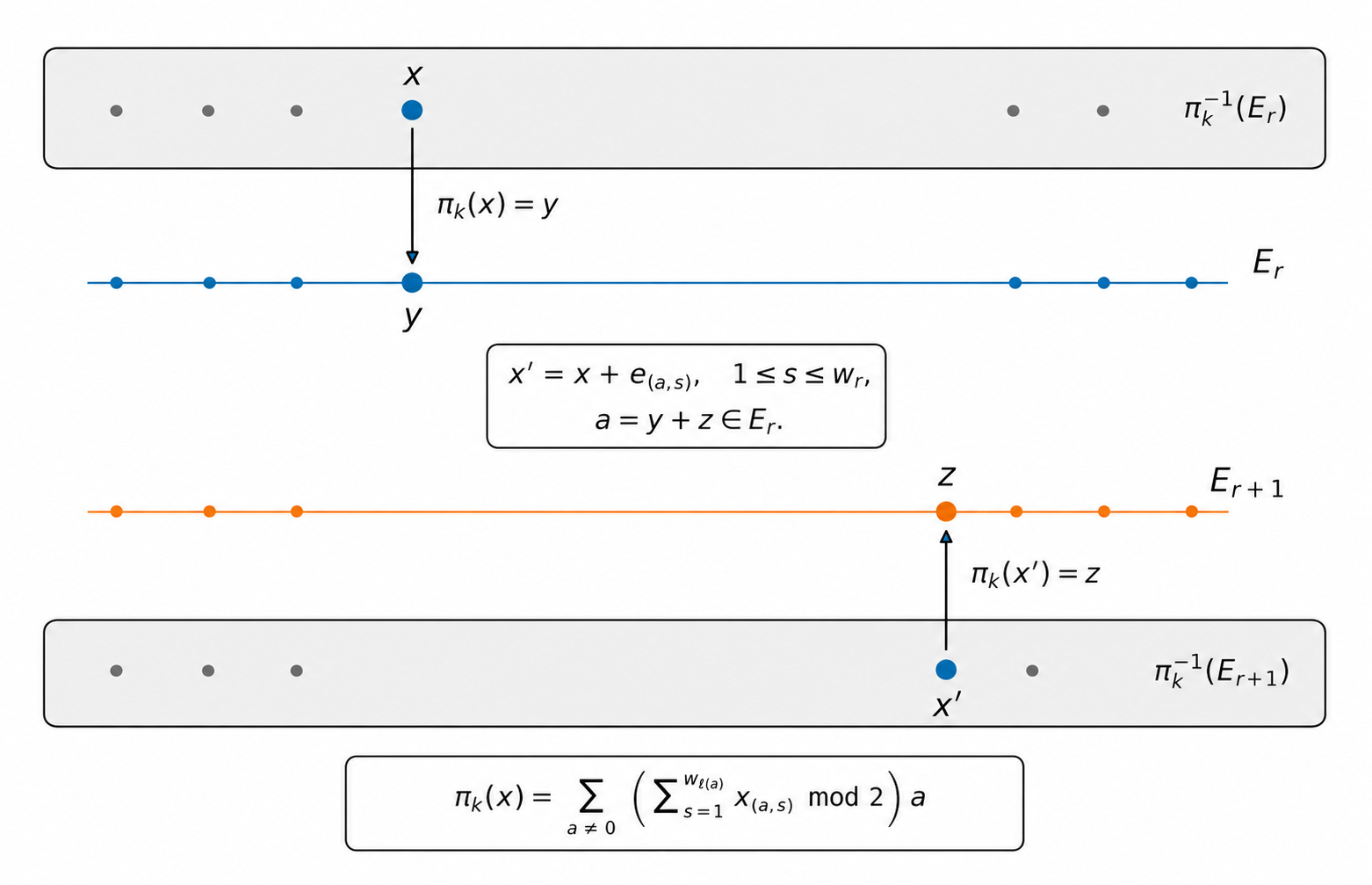}
   \caption{The four layers in the quotient-cube construction.  The vertical
   arrows show the projections $x\mapsto y=\pi_k(x)$ and
   $x'\mapsto z=\pi_k(x')$.  For fixed $y\in E_r$ and $z\in E_{r+1}$,
   the label $a=y+z\in E_r$ has multiplicity $w_r$, producing the neighbors
   $x'=x+e_{(a,s)}$, $1\le s\le w_r$.}
   \label{fig:quotient-cube-four-layers}
\end{figure}

\smallskip
\noindent\emph{Lower bound for the optimized fractional boundary.}
For $1\le r\le k$, assign the weight
\[
   \lambda_r=\frac1{\sqrt{3m_{r+1}w_r}}
\]
to every sensitive cube edge whose quotient endpoints lie in $E_r$ and $E_{r+1}$, and assign weight $0$ to all other sensitive edges.  Consecutive levels have opposite signs, so every selected edge is sensitive.

We verify the relevant degree counts.  Fix $z\in E_r$ and $y\in E_{r+1}$.  The unique quotient label that sends $z$ to $y$ is $a=z+y$.  Its first nonzero coordinate is $r$, so $a\in E_r$ and it occurs among the cube coordinates with multiplicity $w_r$.  Hence every quotient point in $E_r$ has exactly $m_{r+1}w_r$ selected coordinate directions to $E_{r+1}$.  Similarly, for $r\ge2$, every point in $E_r$ has exactly $m_{r-1}w_{r-1}$ selected directions to $E_{r-1}$.

At a vertex whose quotient lies in $E_r$, the selected edges going to $E_{r+1}$ contribute
\[
   m_{r+1}w_r\lambda_r^2=\frac13
\]
to the sum of squared incident weights.  For $2\le r\le k$, the selected edges going to $E_{r-1}$ contribute
\[
   m_{r-1}w_{r-1}\lambda_{r-1}^2
   =\frac{m_{r-1}}{3m_r}
   =\frac23.
\]
At the first level only the first contribution is present.  At the zero quotient $E_{k+1}=\{0\}$, the selected edges to $E_k$ contribute
\[
   m_kw_k\lambda_k^2=\frac13.
\]
Hence the hypothesis of Lemma \ref{lem:fractional-packing} is satisfied at every cube vertex.

Every fiber of $\pi_k$ has size $2^{n_k-k}$.  Counting each selected edge from its endpoint over $E_r$, the number of selected cube edges between the fibers over $E_r$ and $E_{r+1}$ is
\[
   2^{n_k-k}m_rm_{r+1}w_r.
\]
Lemma \ref{lem:fractional-packing} therefore yields
\begin{align}
   \Tfrac(f_k)
   \ge \frac1{2^k}\sum_{r=1}^k
      m_rm_{r+1}w_r\lambda_r \notag
   =\frac1{2^k\sqrt3}\sum_{r=1}^k
      m_r\sqrt{m_{r+1}w_r}.
      \label{eq:Tfrac-quotient-lower}
\end{align}
For $r<k$,
\[
   m_r\sqrt{m_{r+1}w_r}
   =2^{k-r}\sqrt{2^{k-r-1}2^{3r}}
   =2^{(3k-1)/2},
\]
and for $r=k$ the term equals $2^{3k/2}$, which is larger.  Therefore
\begin{equation}\label{eq:Tfrac-quotient-final}
   \Tfrac(f_k)\ge \frac{k}{\sqrt6}\,2^{k/2}.
\end{equation}

\smallskip
\noindent\emph{Upper bound for the Beckmann boundary.}
We construct an admissible vector field directly.  Define the quotient charge
$q:G_k\to\mathbb R$ by
\begin{equation}\label{eq:quotient-charge}
   q(z)=\frac12\sum_{a\ne0}w_{\ell(a)}
   \bigl(g_k(z)-g_k(z+a)\bigr).
\end{equation}
Since each label $a$ occurs among the cube coordinates with multiplicity
$w_{\ell(a)}$, we have
\begin{equation}\label{eq:Lf-quotient-charge}
   Lf_k(x)=q(\pi_k(x)).
\end{equation}
Moreover,
\begin{equation}\label{eq:quotient-charge-zero-sum}
   \sum_{z\in G_k}q(z)=0,
\end{equation}
because for every fixed $a$, translation by $a$ permutes $G_k$.

For $1\le r\le k$, let $d_r$ be the number of sensitive coordinate directions
at a quotient point $z\in E_r$, counted with label multiplicity.  This number
depends only on $r$.  Indeed, a label in $E_j$ with $j<r$ sends $z$ to $E_j$,
and the direction is sensitive exactly when $j\not\equiv r\pmod2$; a label in
$E_j$ with $j>r$ leaves the quotient in $E_r$ and is not sensitive.  Finally,
labels in $E_r$ send $z$ to the levels $E_s$, $s>r$, with exactly $m_s$ labels
leading to $E_s$, each occurring with multiplicity $w_r$.  Consequently,
\begin{equation}\label{eq:dr-formula}
   d_r=
   \sum_{\substack{j<r\\ j\not\equiv r\ ({\rm mod}\ 2)}}m_jw_j
   +w_r\sum_{\substack{s>r\\ s\not\equiv r\ ({\rm mod}\ 2)}}m_s,
\end{equation}
where the second sum includes $s=k+1$.  The two unrestricted sums satisfy
\[
\begin{aligned}
   \sum_{j<r}m_jw_j
   &=2^k\sum_{j=1}^{r-1}4^j
   <\frac13\,2^{k+2r}
   =\frac13m_rw_r,\\
   \sum_{s=r+1}^{k+1}m_s
   &=\sum_{s=r+1}^{k}2^{k-s}+1
   =2^{k-r}=m_r.
\end{aligned}
\]
Therefore
\begin{equation}\label{eq:dr-upper}
   d_r\le\frac43m_rw_r.
\end{equation}
Every sensitive direction at a point over $E_r$ contributes $(-1)^r$ to
$Lf_k$, and hence
\begin{equation}\label{eq:q-on-level}
   q(z)=(-1)^rd_r
   \qquad(z\in E_r).
\end{equation}

For a cube coordinate $(a,s)\in I_k$, define
\begin{equation}\label{eq:direct-vector-field}
   \V_{(a,s)}(x)
   =-\frac{q(a)}{w_{\ell(a)}}\,
   \1_{\{\pi_k(x)=0\}}.
\end{equation}
We check its divergence.  Suppose first that $\pi_k(x)=z\ne0$.  Then
$\V_{(a,s)}(x)=0$ for every $(a,s)$, and the neighbor
$x+e_{(a,s)}$ lies in the zero fiber exactly when $a=z$.  Hence
\[
\begin{aligned}
   \divg\V(x)
   &=-\sum_{s=1}^{w_{\ell(z)}}
      \V_{(z,s)}(x+e_{(z,s)})
     =q(z).
\end{aligned}
\]
If $\pi_k(x)=0$, every neighbor $x+e_{(a,s)}$ has nonzero quotient $a$, so
\[
\begin{aligned}
   \divg\V(x)
   &=\sum_{a\ne0}\sum_{s=1}^{w_{\ell(a)}}
      \V_{(a,s)}(x)
     =-\sum_{a\ne0}q(a)
     =q(0),
\end{aligned}
\]
where the last equality follows from \eqref{eq:quotient-charge-zero-sum}.
Together with \eqref{eq:Lf-quotient-charge}, this proves
$\divg\V=Lf_k$.

The field is supported on the zero fiber.  At every point in that fiber,
\begin{align*}
   \|\V(x)\|_2^2
   =\sum_{a\ne0}\sum_{s=1}^{w_{\ell(a)}}
      \frac{q(a)^2}{w_{\ell(a)}^2}
   =\sum_{r=1}^k\frac{m_rd_r^2}{w_r},
\end{align*}
where we used \eqref{eq:q-on-level}.  Since $\pi_k$ is surjective, a uniform
cube point lies in the zero fiber with probability $2^{-k}$.  Thus, by the
definition of $\B$ and \eqref{eq:dr-upper},
\begin{align*}
   \B(f_k)
   \le \E\|\V(x)\|_2
    =2^{-k}\left(\sum_{r=1}^k\frac{m_rd_r^2}{w_r}\right)^{1/2}
   \le\frac43\,2^{-k}
      \left(\sum_{r=1}^k m_r^3w_r\right)^{1/2}.
\end{align*}
Since
\[
   m_r^3w_r=2^{3(k-r)}2^{3r}=2^{3k}
\]
for every $1\le r\le k$, we obtain
\begin{equation}\label{eq:B-quotient-upper}
   \B(f_k)\le\frac43\sqrt{k}\,2^{k/2}.
\end{equation}
Combining \eqref{eq:Tfrac-quotient-final} and \eqref{eq:B-quotient-upper} proves \eqref{eq:unbounded-ratio}.  Finally, $\Tfrac(f_k)\le\Ccol(f_k)$, while
\[
   2^{3k}\le n_k\le\frac43\,2^{3k}
\]
by \eqref{eq:nk-dimension}.  Thus $k\asymp\log n_k$, which proves the remaining assertions.
\end{proof}

\begin{remark}\label{rem:unbounded-mechanism}
The mechanism is genuinely nonlocal.  Fractional and integral colorings must pay for sensitive edges at their endpoints.  By contrast, the Beckmann boundary prescribes only the divergence and may route the resulting charge through nonsensitive edges.  The quotient construction creates $k$ edge scales that add linearly in the fractional packing lower bound, while the direct Beckmann construction combines them in an $\ell_2$ fashion.  This produces the factor $\sqrt{k}$.
\end{remark}

\subsection{Summary of comparisons}\label{summary}
For every $f$, every fixed coloring $\chi$, and every fixed fractional coloring $\theta$,
\[
\begin{gathered}
   \B(f)\le\Tfrac(f)\le\Ccol(f)\le\Tcol^\chi(f),
   \qquad \Tfrac(f)\le\Ttheta(f),\\
   \B(f)\le\E\sqrt{h_{\{f=1\}}}.
\end{gathered}
\]
The reverse inequalities fail quantitatively.  Proposition \ref{prop:singleton} shows that fixed colorings, fixed fractional colorings and one-sided boundaries may be larger than $\B(f)$ by $\sqrt n$.  Theorem \ref{thm:unbounded-separation} shows that even the optimized fractional and integral colored boundaries satisfy
\[
   \sup_{\substack{f:\Omega_{n_k}\to\{-1,1\}\\ f\ \mathrm{nonconstant}}}
   \frac{\Tfrac(f)}{\B(f)}\gtrsim\sqrt{\log n_k},
   \qquad
   \sup_{\substack{f:\Omega_{n_k}\to\{-1,1\}\\ f\ \mathrm{nonconstant}}}
   \frac{\Ccol(f)}{\B(f)}\gtrsim\sqrt{\log n_k}
\]
along the explicit sequence of dimensions $(n_k)$.  Thus no dimension-free reverse comparison is possible, and the Beckmann--Talagrand theorem is genuinely stronger than the optimized colored-boundary inequality. As a complement to these lower bounds, Corollary \ref{log-reverse-beckmann-fractional} proved that for every Boolean function \(f:\Omega_n\to\{-1,1\}\),
\[
   \Tfrac(f)\lesssim \B(f)\log(en),
\]
and 
\[
\Tfrac(f)\asymp\Ccol(f)
\]
was established in Proposition \ref{prop:fractional-integral-comparable}.
\section{The IVHV identity and strong spectral endpoints}\label{sec:ivhv-spectral}

The proof of Theorem \ref{thm:beckmann-Poincar\'e} used the derivative identity of Ivanisvili--van Handel--Volberg \cite[Lemma~2.1]{IVHV}, which lies behind the reverse Poincar\'e estimate, only through Cauchy--Schwarz.  In this section we keep the exact kernel.  This gives strong spectral moment inequalities, including the endpoint estimate \eqref{eq:intro-strong-endpoint}.

For $t>0$, put $\rho=e^{-t}$ and let $\varepsilon_i(t)$ be independent random signs with
\[
   \E\varepsilon_i(t)=\rho.
\]
Define the standardized biased Rademacher variables
\[
   \delta_i(t)=\frac{\varepsilon_i(t)-\rho}{\sqrt{1-\rho^2}}.
\]
For $a=(a_1,\ldots,a_n)\in\R^n$ define the IVHV local gauge
\begin{equation}\label{eq:Psi-def}
   \Psi_t(a)=
   \frac{1}{\sqrt{e^{2t}-1}}
   \E\left|\sum_{i=1}^n a_i\delta_i(t)\right|.
\end{equation}

\begin{lemma}[Heat-kernel derivative identity]\label{lem:ivhv-identity}
For every $g:\Omega_n\to\R$, every $t>0$, every $x\in\Omega_n$ and every coordinate $i$,
\begin{equation}\label{eq:ivhv-identity}
   D_iP_tg(x)=
   \frac{1}{\sqrt{e^{2t}-1}}
   \E\big[\delta_i(t)g(x\varepsilon(t))\big].
\end{equation}
\end{lemma}

The identity is \cite[Lemma~2.1]{IVHV}.  A verification is recalled in Appendix \ref{app:heat-kernel}.

We first state the multiplier form for driftless Bernstein functions.  Thus $\phi(0)=0$, the linear drift coefficient vanishes, and the L\'evy--Khintchine representation \cite[Theorem~3.2]{SchillingSongVondracek} is
\begin{equation}\label{eq:bernstein-rep}
   \phi(\lambda)=\int_0^\infty (1-e^{-t\lambda})\,\nu(\dd t),
\end{equation}
where $\nu$ is a positive measure satisfying
\[
   \int_0^\infty (1\wedge t)\,\nu(\dd t)<\infty.
\]
A general Bernstein function with $\phi(0)=0$ may additionally contain a drift term $b\lambda$; that term is not included in the statement below.  Put
\[
   \eta_\phi(s)=\nu([s,\infty))
\]
and define
\begin{equation}\label{eq:Gamma-phi-def}
   \Gamma_\phi(a)=2\int_0^\infty \eta_\phi(s)\Psi_s(a)\,\dd s,
\end{equation}
with the convention that the value may be $+\infty$.

\begin{theorem}[Driftless Beckmann multiplier inequality]\label{thm:bernstein-multiplier}
Let $f:\Omega_n\to[-1,1]$.  If $\V$ is a vector field satisfying $\divg \V=Lf$, then
\begin{equation}\label{eq:bernstein-multiplier}
   \ip{f}{\phi(L)f}
   \le
   \E_x\Gamma_\phi(\V(x)).
\end{equation}
Consequently,
\[
   \ip{f}{\phi(L)f}
   \le
   \inf_{\divg \V=Lf}\E_x\Gamma_\phi(\V(x)).
\]
\end{theorem}

\begin{proof}
For $t>0$ set
\[
   \operatorname{NS}_t(f)=\ip{f}{(I-P_t)f}.
\]
Let $\V$ be admissible.  As in the proof of Theorem \ref{thm:beckmann-Poincar\'e},
\[
   \operatorname{NS}_t(f)
   =2\int_0^t\E_x\sum_{i=1}^n \V_i(x)D_iP_sf(x)\,\dd s.
\]
Using Lemma \ref{lem:ivhv-identity} and $|f|\le1$,
\[
\begin{aligned}
   \left|\E_x\sum_i \V_i(x)D_iP_sf(x)\right|
   &\le
   \frac{1}{\sqrt{e^{2s}-1}}
   \E_{x,\varepsilon}\left|\sum_i\V_i(x)\delta_i(s)f(x\varepsilon(s))\right|  \\
   &\le
   \frac{1}{\sqrt{e^{2s}-1}}
   \E_{x,\varepsilon}\left|\sum_i\V_i(x)\delta_i(s)\right|  \\
   &=\E_x\Psi_s(\V(x)).
\end{aligned}
\]
Therefore
\begin{equation}\label{eq:NS-Psi}
   \operatorname{NS}_t(f)
   \le
   2\E_x\int_0^t\Psi_s(\V(x))\,\dd s.
\end{equation}
By the spectral theorem and \eqref{eq:bernstein-rep},
\[
   \ip{f}{\phi(L)f}
   =\int_0^\infty \operatorname{NS}_t(f)\,\nu(\dd t).
\]
The integrand on the right-hand side of \eqref{eq:NS-Psi} is nonnegative.  Hence Tonelli's theorem gives
\[
\begin{aligned}
   \ip{f}{\phi(L)f}
   &\le
   2\E_x\int_0^\infty\int_0^t\Psi_s(\V(x))\,\dd s\,\nu(\dd t)  \\
   &=
   2\E_x\int_0^\infty\nu([s,\infty))\Psi_s(\V(x))\,\dd s \\
   &=
   \E_x\Gamma_\phi(\V(x)).
\end{aligned}
\]
This proves the theorem.
\end{proof}

We now specialize to fractional powers.  For $0<\alpha<1$, the standard Bernstein representation \cite[Chapter~3]{SchillingSongVondracek} gives
\begin{equation}\label{eq:power function rep}
       \lambda^\alpha=c_\alpha\int_0^\infty(1-e^{-t\lambda})t^{-1-\alpha}\,\dd t.
\end{equation}
Thus Theorem \ref{thm:bernstein-multiplier} gives the following.

\begin{corollary}[Fractional spectral moments]\label{cor:fractional-multiplier}
Let $0<\alpha<1$ and define
\begin{equation}\label{eq:Gamma-alpha-def}
   \Gamma_\alpha(a)=\int_0^\infty s^{-\alpha}\Psi_s(a)\,\dd s.
\end{equation}
Then for every $f:\Omega_n\to[-1,1]$,
\begin{equation}\label{eq:fractional-general}
   \sum_{S\ne\varnothing}|S|^\alpha\widehat f(S)^2
   =\ip{f}{L^\alpha f}
   \lesssim_\alpha
   \inf_{\divg \V=Lf}\E\Gamma_\alpha(\V(x)).
\end{equation}
\end{corollary}

For completeness, we record the exact identity behind \eqref{eq:fractional-general}.  Let $\V$ be any field satisfying $\divg\V=Lf$ and put
\[
   c_\alpha=\frac{\alpha}{\Gamma(1-\alpha)}.
\]
Using \eqref{eq:power function rep} and the identity
\[
   \ip{f}{(I-P_t)f}
   =2\int_0^t\E_x\sum_i\V_i(x)D_iP_sf(x)\,\dd s,
\]
we obtain
\[
\begin{aligned}
   \ip{f}{L^\alpha f}
   &=2c_\alpha\int_0^\infty\int_0^t
      \E_x\sum_i\V_i(x)D_iP_sf(x)\,\dd s\,t^{-1-\alpha}\,\dd t\\
   &=\frac{2}{\Gamma(1-\alpha)}
      \int_0^\infty s^{-\alpha}
      \E_x\sum_i\V_i(x)D_iP_sf(x)\,\dd s\\
   &=\frac{2}{\Gamma(1-\alpha)}
      \int_0^\infty
      \E_{x,\varepsilon}
      \left[
         f(x\varepsilon(s))
         \sum_{i=1}^n \V_i(x)\delta_i(s)
      \right]
      \frac{s^{-\alpha}}{\sqrt{e^{2s}-1}}\,\dd s.
\end{aligned}
\]
Here the interchange of the $s$- and $t$-integrals is justified absolutely.  Indeed, the finite Walsh expansion gives
\[
   D_iP_sf=\sum_{S\ni i}e^{-s|S|}\widehat f(S)\chi_S.
\]
Consequently $D_iP_sf=O(1)$ as $s\downarrow0$ and $D_iP_sf=O(e^{-s})$ as $s\to\infty$, with constants depending only on the fixed functions $f$ and $\V$.  Thus
\[
   s^{-\alpha}\left|\E_x\sum_i\V_iD_iP_sf\right|
\]
is integrable on $(0,\infty)$ because $0<\alpha<1$.  Taking absolute values in the exact identity and using $|f|\le1$ recovers Corollary \ref{cor:fractional-multiplier}.

The next elementary estimate converts the exact IVHV gauge into local boundary moments.

\begin{lemma}[Flat-vector estimate]\label{lem:flat-vector}
Let $u_m=(1,\ldots,1,0,\ldots,0)\in\R^n$ have $m$ nonzero coordinates.  For $0<\alpha<1$,
\begin{equation}\label{eq:flat-Gamma-bound}
   \Gamma_\alpha(u_m)
   \lesssim_\alpha
   \omega_\alpha(m),
\end{equation}
where $\omega_\alpha(0)=0$ and, for $m\ge1$,
\begin{equation}\label{eq:omega-alpha-def}
   \omega_\alpha(m)=
   \begin{cases}
      \sqrt m, & 0<\alpha<1/2,\\[3pt]
      \sqrt m\log(e+m), & \alpha=1/2,\\[3pt]
      m^\alpha, & 1/2<\alpha<1.
   \end{cases}
\end{equation}
The same bound holds for every vector $a$ with at most $m$ nonzero coordinates and $|a_i|\le1$.
\end{lemma}

\begin{proof}
It suffices to prove the estimate for such a vector $a$.  If $m=0$, then $a=0$, so $\Psi_t(a)=\Gamma_\alpha(a)=0$, and the assertion is immediate.  Assume henceforth that $m\ge1$.

First suppose $0<t\le1$.  The variables $\delta_i(t)$ are centered and have variance one.  Hence Cauchy--Schwarz gives
\[
   \E\left|\sum_i a_i\delta_i(t)\right|
   \le \left(\sum_i a_i^2\right)^{1/2}
   \le \sqrt m.
\]
Since $e^{2t}-1\asymp t$ on $(0,1]$,
\begin{equation}\label{eq:Psi-L2-small}
   \Psi_t(a)\lesssim \sqrt{m/t}.
\end{equation}
We also need an $\ell_1$-type bound.  By the triangle inequality and $|a_i|\le1$, it is enough to estimate $\Psi_t(e_i)$.  If $\rho=e^{-t}$, then
\[
   \E|\delta_i(t)|=\sqrt{1-\rho^2},
   \qquad
   \Psi_t(e_i)=\frac{\sqrt{1-\rho^2}}{\sqrt{e^{2t}-1}}=\rho\le1.
\]
Therefore
\begin{equation}\label{eq:Psi-L1-small}
   \Psi_t(a)\le m,
   \qquad 0<t\le1.
\end{equation}
Combining \eqref{eq:Psi-L2-small} and \eqref{eq:Psi-L1-small},
\begin{equation}\label{eq:Psi-min}
   \Psi_t(a)
   \lesssim
   \min\left\{m,\sqrt{m/t}\right\},
   \qquad 0<t\le1.
\end{equation}
For $t\ge1$, the prefactor in \eqref{eq:Psi-def} is $O(e^{-t})$, while the centered sum has $L^1$ norm at most its $L^2$ norm, which is at most $\sqrt m$.  Hence
\begin{equation}\label{eq:Psi-large}
   \Psi_t(a)\lesssim e^{-t}\sqrt m,
   \qquad t\ge1.
\end{equation}
The contribution of $t\ge1$ to $\Gamma_\alpha(a)$ is therefore $O_\alpha(\sqrt m)$.

For $0<t\le1$, split at $t=1/m$.  By \eqref{eq:Psi-min},
\[
\begin{aligned}
   \int_0^1 t^{-\alpha}\Psi_t(a)\,\dd t
   &\lesssim
   m\int_0^{1/m}t^{-\alpha}\,\dd t
   +\sqrt m\int_{1/m}^1t^{-\alpha-1/2}\,\dd t\\
   &=\frac{m^\alpha}{1-\alpha}
   +\sqrt m\int_{1/m}^1t^{-\alpha-1/2}\,\dd t.
\end{aligned}
\]
The second term equals, up to a constant depending only on $\alpha$,
\[
   \begin{cases}
      \sqrt m\bigl(1-m^{\alpha-1/2}\bigr), & 0<\alpha<1/2,\\
      \sqrt m\log m, & \alpha=1/2,\\
      \sqrt m\bigl(m^{\alpha-1/2}-1\bigr), & 1/2<\alpha<1.
   \end{cases}
\]
Thus it is bounded respectively by $C_\alpha\sqrt m$, $\sqrt m\log(e+m)$, and $C_\alpha m^\alpha$.  The first integral has size $C_\alpha m^\alpha$; here $m^\alpha\le\sqrt m$ when $\alpha<1/2$, while $\sqrt m\le m^\alpha$ when $\alpha>1/2$.  Combining these estimates with the $t\ge1$ contribution proves \eqref{eq:flat-Gamma-bound} in all three regimes.
\end{proof}

\begin{theorem}[One-sided spectral moment scale]\label{thm:one-sided-moment-scale}
Let $\A\subset\Omega_n$ and
\[
   h_{\A}(x)=\#\{i:x\in\A,\ x^{\oplus i}\notin\A\}.
\]
Then for every $0<\alpha<1$,
\begin{equation}\label{eq:one-sided-alpha}
   \sum_{S\ne\varnothing}|S|^\alpha\widehat{\1_{\A}}(S)^2
   \lesssim_\alpha
   \E\omega_\alpha(h_{\A}(x)).
\end{equation}
In particular, at the critical endpoint $\alpha=1/2$,
\begin{equation}\label{eq:one-sided-half}
   \sum_{S\ne\varnothing}\sqrt{|S|}\,\widehat{\1_{\A}}(S)^2
   \lesssim
   \E\big[\sqrt{h_{\A}(x)}\log(e+h_{\A}(x))\big].
\end{equation}
\end{theorem}

\begin{proof}
Apply Corollary \ref{cor:fractional-multiplier} to $g=\1_{\A}$.  Define a vector field $\V^+$ by
\[
   \V_i^+(x)=\frac12\,\1_{\{x\in\A,\ x^{\oplus i}\notin\A\}}.
\]
Then $\divg \V^+=L\1_{\A}$.  Indeed, on a boundary edge from $\A$ to $\A^c$, the difference $\V_i^+(x)-\V_i^+(x^{\oplus i})$ equals $1/2$ at the point in $\A$ and $-1/2$ at the point in $\A^c$, exactly matching $D_i\1_{\A}$.  On non-boundary edges both sides vanish.

At a point $x\in\A$, the vector $\V^+(x)$ has $h_{\A}(x)$ nonzero coordinates, all equal to $1/2$; at a point $x\notin\A$ it is zero.  By homogeneity of $\Gamma_\alpha$ and Lemma \ref{lem:flat-vector},
\[
   \Gamma_\alpha(\V^+(x))
   \lesssim_\alpha
   \omega_\alpha(h_{\A}(x)).
\]
Substituting this admissible field into \eqref{eq:fractional-general} proves \eqref{eq:one-sided-alpha}, and \eqref{eq:one-sided-half} is the case $\alpha=1/2$.
\end{proof}

The same argument gives a colored version.

\begin{corollary}[Colored strong spectral moments]\label{cor:colored-moment-scale}
Let $f:\Omega_n\to\{-1,1\}$ and color every sensitive edge red or blue.  Then for every $0<\alpha<1$,
\begin{equation}\label{eq:colored-alpha}
   \sum_{S\ne\varnothing}|S|^\alpha\widehat f(S)^2
   \lesssim_\alpha
   \E\left[\omega_\alpha(s_{f,{\rm red}}(x))+\omega_\alpha(s_{f,{\rm blue}}(x))\right].
\end{equation}
In particular,
\begin{align*}
   \sum_{S\ne\varnothing}\sqrt{|S|}\,\widehat f(S)^2
   &\lesssim
   \E\Big[
      \sqrt{s_{f,{\rm red}}(x)}\log(e+s_{f,{\rm red}}(x)) \\
   &\hspace{3.2cm}+
      \sqrt{s_{f,{\rm blue}}(x)}\log(e+s_{f,{\rm blue}}(x))
   \Big].
\end{align*}
\end{corollary}

\begin{proof}
For every sensitive edge $e=\{x,x^{\oplus i}\}$, orient the notation so that $f(x)=1$ and $f(x^{\oplus i})=-1$.  Define the two endpoint values of the $i$th field component by
\[
\begin{array}{c|cc}
 &x\text{, the }f=1\text{ endpoint}&x^{\oplus i}\text{, the }f=-1\text{ endpoint}\\ \hline
 e\text{ red}&1&0\\
 e\text{ blue}&0&-1.
\end{array}
\]
On nonsensitive edges set both endpoint values equal to zero.  In either color case,
\[
   \V_i(x)-\V_i(x^{\oplus i})=1=D_i f(x),
\]
and at the opposite endpoint the same identity holds with both sides negated.  Summing over $i$ therefore gives $\divg\V=Lf$.

If $f(x)=1$, the nonzero coordinates of $\V(x)$ are precisely the red sensitive edges charged at $x$, and all those entries equal $1$.  If $f(x)=-1$, the nonzero coordinates are precisely the blue sensitive edges charged at $x$, and all those entries equal $-1$.  Lemma \ref{lem:flat-vector} applies to the absolute values of these vectors, and Corollary \ref{cor:fractional-multiplier} gives \eqref{eq:colored-alpha}.
\end{proof}

\begin{lemma}[Uniform binomial ratio]\label{lem:uniform-binomial-ratio}
There are universal constants $c,C>0$ such that for every integer $m\ge0$ and every $0\le j\le m$,
\begin{equation}\label{eq:uniform-binomial-ratio}
   c\sqrt{\frac{m+1}{(j+1)(m-j+1)}}
   \le
   \frac{\binom mj^2}{\binom{2m}{2j}}
   \le
   C\sqrt{\frac{m+1}{(j+1)(m-j+1)}}.
\end{equation}
\end{lemma}

\begin{proof}
The exact identity
\[
   \frac{\binom mj^2}{\binom{2m}{2j}}
   =\frac{\binom{2j}{j}\binom{2(m-j)}{m-j}}{\binom{2m}{m}}
\]
follows by expanding the three binomial coefficients into factorials.  Two-sided Stirling estimates imply that there are universal constants $c_0,C_0>0$ such that, for every integer $q\ge1$,
\[
   c_0\frac{4^q}{\sqrt q}
   \le \binom{2q}{q}
   \le C_0\frac{4^q}{\sqrt q}.
\]
After changing the constants, the same statement holds for all $q\ge0$ with $\sqrt q$ replaced by $\sqrt{q+1}$; the case $q=0$ is immediate.  Applying these bounds to $q=j$, $q=m-j$, and $q=m$, the powers of $4$ cancel and yield \eqref{eq:uniform-binomial-ratio}, including the endpoint cases $j=0$ and $j=m$.
\end{proof}

\begin{proposition}[Sharpness for majority]\label{prop:majority-sharp}
Let $n\ge3$ be odd and let $\A_n=\{x\in\Omega_n:\sum_i x_i>0\}$.  Then, for $0<\alpha<1$,
\[
   \sum_{S\ne\varnothing}|S|^\alpha\widehat{\1_{\A_n}}(S)^2
   \asymp_\alpha
   \E\omega_\alpha(h_{\A_n}(x)).
\]
Consequently the logarithmic endpoint \eqref{eq:one-sided-half} is sharp up to universal constants.
\end{proposition}

\begin{proof}
The boundary of $\A_n$ is supported on the positive critical layer
\[
   \#\{i:x_i=1\}=\frac{n+1}{2}.
\]
On this layer,
\[
   h_{\A_n}(x)=\frac{n+1}{2},
\]
and the layer has measure comparable to $n^{-1/2}$ by the standard central-binomial estimates.  Hence
\begin{equation}\label{eq:majority-boundary-scale}
   \E\omega_\alpha(h_{\A_n})
   \asymp_\alpha
   \begin{cases}
      1, & 0<\alpha<1/2,\\
      \log n, & \alpha=1/2,\\
      n^{\alpha-1/2}, & 1/2<\alpha<1.
   \end{cases}
\end{equation}

It remains to obtain a finite-dimensional lower bound for the Fourier moment.  Write $n=2m+1$.  Majority is odd under the global sign change $x\mapsto-x$; since $\chi_S(-x)=(-1)^{|S|}\chi_S(x)$, every positive even Fourier level of majority vanishes.  By permutation symmetry, all Fourier coefficients on a fixed odd level agree.  For $|S|=2j+1$, the exact majority formula is \cite[Section~5.3]{ODonnell}
\begin{equation}\label{eq:majority-exact-coefficient}
   \widehat{\mathrm{Maj}_n}(S)
   =(-1)^j
     \frac{\binom{2m}{m}}{2^{2m}}
     \frac{\binom{m}{j}}{\binom{2m}{2j}},
   \qquad 0\le j\le m.
\end{equation}
Consequently,
\begin{equation}\label{eq:majority-level-formula}
\begin{aligned}
   W_{=2j+1}[\mathrm{Maj}_n]
   &=\binom{2m+1}{2j+1}
     \left(\frac{\binom{2m}{m}}{2^{2m}}
     \frac{\binom{m}{j}}{\binom{2m}{2j}}\right)^2\\
   &=\frac{2m+1}{2j+1}
     \left(\frac{\binom{2m}{m}}{2^{2m}}\right)^2
     \frac{\binom{m}{j}^2}{\binom{2m}{2j}}.
\end{aligned}
\end{equation}
The central-binomial estimate used in Lemma \ref{lem:uniform-binomial-ratio} gives
\[
   \left(\frac{\binom{2m}{m}}{2^{2m}}\right)^2
   \gtrsim \frac1{m+1}.
\]
Combining this with Lemma \ref{lem:uniform-binomial-ratio} in \eqref{eq:majority-level-formula},
\[
\begin{aligned}
   W_{=2j+1}[\mathrm{Maj}_n]
   &\gtrsim
   \frac{2m+1}{(2j+1)(m+1)}
   \sqrt{\frac{m+1}{(j+1)(m-j+1)}}\\
   &\gtrsim
   \frac1{(2j+1)\sqrt{j+1}}
   \gtrsim (2j+1)^{-3/2}.
\end{aligned}
\]
Thus
\begin{equation}\label{eq:majority-uniform-lower}
   W_{=k}[\mathrm{Maj}_n]\gtrsim k^{-3/2}
   \qquad\text{for every odd }1\le k\le n.
\end{equation}
Since $\1_{\A_n}=(1+\mathrm{Maj}_n)/2$, its nonconstant level weights are one quarter of those of $\mathrm{Maj}_n$.  Summing \eqref{eq:majority-uniform-lower} over odd $k\le n/2$ yields
\[
   \sum_{S\ne\varnothing}|S|^\alpha
      \widehat{\1_{\A_n}}(S)^2
   \gtrsim_\alpha
   \begin{cases}
      1, & 0<\alpha<1/2,\\
      \log n, & \alpha=1/2,\\
      n^{\alpha-1/2}, & 1/2<\alpha<1.
   \end{cases}
\]
The reverse inequality follows directly from Theorem \ref{thm:one-sided-moment-scale} and \eqref{eq:majority-boundary-scale}.  This proves the proposition.
\end{proof}

\begin{remark}[Relation to Majority Is Stablest]\label{rem:majority-is-stablest}
For a balanced function $g:\Omega_n\to\{-1,1\}$, write
\[
   \Stab_\rho(g)=\sum_{S\subseteq[n]}\rho^{|S|}\widehat g(S)^2.
\]
By \eqref{eq:power function rep}, after the change of variables $\rho=e^{-t}$,
\begin{equation}\label{eq:fractional-stability-integral}
   \sum_{S\ne\varnothing}|S|^\alpha\widehat g(S)^2
   =c_\alpha\int_0^1
   \frac{1-\Stab_\rho(g)}{\rho(-\log\rho)^{1+\alpha}}\,\dd\rho,
   \qquad
   c_\alpha=\frac{\alpha}{\Gamma(1-\alpha)}.
\end{equation}
Majority Is Stablest \cite[Theorem~4.4]{MOO} says that, for every fixed
$0<\rho<1$, balanced Boolean functions with vanishing maximal influence have
noise stability at most
\[
   \frac{2}{\pi}\arcsin\rho+o(1),
\]
the limiting stability of majority.  Since
\[
   1-\frac{2}{\pi}\arcsin\rho\asymp\sqrt{1-\rho}
   \qquad(\rho\uparrow1),
\]
the Gaussian expression obtained by inserting this bound into
\eqref{eq:fractional-stability-integral} is finite exactly when
$\alpha<1/2$.  Thus, after truncating the integral away from $\rho=1$ and
passing to the low-influence limit, Majority Is Stablest explains the
constant-order obstruction, and hence the $\sqrt m$ branch of
$\omega_\alpha$, for $\alpha<1/2$.  At and above the critical exponent it
detects divergence, but the usual fixed-$\rho$ theorem does not by itself
supply the sharp finite-dimensional rates $\log n$ and
$n^{\alpha-1/2}$; those require control down to the scale
$1-\rho\asymp n^{-1}$, as in Proposition \ref{prop:majority-sharp}.
\end{remark}

\section{Relation to random restrictions}\label{sec:random-restrictions}

The proof of Eldan--Kindler--Lifshitz--Minzer \cite[Sections~2.2 and~3.4]{EKLM} is based on random restrictions.  Their method first proves a level-one estimate for a boundary functional and then shows that random restrictions turn high-degree Fourier mass into level-one Fourier mass.  For edge-local quantities, the boundary of a restriction is obtained by deleting frozen coordinates, and this gives the required contraction.

For the Beckmann boundary, the level-one estimate is immediate.  If $\divg \V=Lf$, then, for every $i$,
\[
\begin{aligned}
   \widehat f(\{i\})
   &=\ip{f}{\chi_i}
     =\ip{Lf}{\chi_i}
     =\ip{\divg \V}{\chi_i} \\
   &=2\sum_{j=1}^n\ip{\V_j}{D_j\chi_i}
     =2\ip{\V_i}{\chi_i}.
\end{aligned}
\]
Here we used $L\chi_i=\chi_i$ and $D_j\chi_i=\1_{\{i=j\}}\chi_i$.  Therefore, by the vector-valued Jensen inequality,
\[
\begin{aligned}
   \sqrt{W_{=1}[f]}
   &=2\left\|\E_x\big(\V_i(x)\chi_i(x)\big)_{i=1}^n\right\|_2\\
   &\le2\E_x\left\|\big(\V_i(x)\chi_i(x)\big)_{i=1}^n\right\|_2
     =2\E\|\V(x)\|_2.
\end{aligned}
\]
Taking the infimum over all admissible fields gives
\[
   \B(f)\ge\frac12\sqrt{W_{=1}[f]}.
\]

The missing random-restriction step would be a Hodge-type contraction estimate of the form
\[
   \E_R\B(f_R)\lesssim \sqrt p\,\B(f).
\]
Here $R=(J,z)$ is sampled as follows.  Each coordinate is retained in $J\subseteq[n]$ independently with probability $p$, and, conditionally on $J$, the vector $z\in\{-1,1\}^{[n]\setminus J}$ is uniform and independent.  The restricted function $f_R$ is the Boolean function on $\Omega_J=\{-1,1\}^J$ obtained by setting the coordinates outside $J$ equal to $z$, and $\E_R$ denotes expectation over both $J$ and $z$.

Unlike the corresponding estimate for edge-local colored boundaries, the displayed contraction does not follow by deleting coordinates from an optimal field.  Indeed, if $\V$ satisfies $\divg \V=Lf$ and one restricts to a slice, then the divergence of the components in the alive directions equals the restricted alive-coordinate Laplacian only after subtracting the divergence contributed by the frozen directions.  That latter term is generally nonzero and becomes an additional source on the slice.  Thus the alive part of $\V$ need not be feasible for $f_R$.  The semigroup proof avoids this obstruction: it tests the full divergence equation directly against $P_s f$ and never restricts the vector field.

This is the sense in which Theorem \ref{thm:main} strengthens the robust theorem of Eldan--Kindler--Lifshitz--Minzer by a semigroup method.  Their colored boundary is edge-local; the Beckmann boundary is nonlocal and smaller, yet it satisfies the same Talagrand-type lower bound.

\appendix

\section{Standard analytic and duality ingredients}\label{app:standard}

\subsection{Heat-kernel identities and reverse Poincar\'e}\label{app:heat-kernel}

\begin{proof}[Proof of Lemma \ref{lem:ivhv-identity}]
Write $\rho=e^{-t}$ and fix $x$.  In the noise representation
\[
   P_tg(x)=\E_\varepsilon g(x\varepsilon),
\]
the coordinates of $\varepsilon$ are independent and satisfy
\[
   \Pr[\varepsilon_i=1]=\frac{1+\rho}{2},
   \qquad
   \Pr[\varepsilon_i=-1]=\frac{1-\rho}{2}.
\]
Fix all noise variables except $\varepsilon_i$, and abbreviate
\[
   a=g(x_1\varepsilon_1,\ldots,x_i,\ldots,x_n\varepsilon_n),
   \qquad
   b=g(x_1\varepsilon_1,\ldots,-x_i,\ldots,x_n\varepsilon_n).
\]
Conditioning on $\varepsilon_{\ne i}$ gives
\[
   P_tg(x)-P_tg(x^{\oplus i})
   =\rho\,\E_{\varepsilon_{\ne i}}(a-b),
\]
and hence
\begin{equation}\label{eq:appendix-DPt-first}
   D_iP_tg(x)=\frac\rho2\E_{\varepsilon_{\ne i}}(a-b).
\end{equation}
On the other hand, a direct two-point computation gives
\[
   \E_{\varepsilon_i}\big[g(x\varepsilon)(\varepsilon_i-\rho)\big]
   =\frac{1-\rho^2}{2}(a-b).
\]
Since $\delta_i(t)=(\varepsilon_i-\rho)/\sqrt{1-\rho^2}$,
\[
   \E_\varepsilon[\delta_i(t)g(x\varepsilon)]
   =\frac{\sqrt{1-\rho^2}}2
      \E_{\varepsilon_{\ne i}}(a-b).
\]
Combining this with \eqref{eq:appendix-DPt-first} and using
\[
   \frac\rho{\sqrt{1-\rho^2}}
   =\frac1{\sqrt{e^{2t}-1}}
\]
proves \eqref{eq:ivhv-identity}.
\end{proof}

We next give a proof of Lemma \ref{lem:reverse-Poincar\'e}. Another semigroup proof can be found in \cite[Section 6]{ILvHV}.

\begin{proof}[Proof of Lemma \ref{lem:reverse-Poincar\'e}]
Write $\rho=e^{-t}$ and use the same noise representation as above.  Fix $x$ and set
\[
   G(\varepsilon)=g(x\varepsilon),
   \qquad
   \delta_i=\frac{\varepsilon_i-\rho}{\sqrt{1-\rho^2}}.
\]
The variables $\delta_1,\ldots,\delta_n$ are orthonormal in the noise
$L^2$ space and are all orthogonal to constants.  The calculation in the proof of Lemma \ref{lem:ivhv-identity} gives
\begin{equation}\label{eq:coef-id}
   D_iP_tg(x)=\frac\rho{\sqrt{1-\rho^2}}\E[G\delta_i].
\end{equation}
Because $\E\delta_i=0$,
\[
   \E[G\delta_i]=\E[(G-\E G)\delta_i].
\]
Applying Bessel's inequality to $G-\E G$ and the orthonormal system
$\delta_1,\ldots,\delta_n$ therefore gives
\[
   \sum_{i=1}^n(\E[G\delta_i])^2
   \le\E[(G-\E G)^2]
   =P_t(g^2)(x)-(P_tg(x))^2.
\]
Using \eqref{eq:coef-id} and
$\rho^2/(1-\rho^2)=1/(e^{2t}-1)$ proves
\eqref{eq:reverse-Poincar\'e}.  The $L^\infty$ consequence follows as in the first proof.
\end{proof}

\subsection{Finite-dimensional Beckmann duality}\label{app:duality}

\begin{proof}[Proof of Proposition \ref{prop:dual}]
Let $X$ be the finite-dimensional space of vector fields on $\Omega_n$, equipped with
\[
   \ip{\V}{W}_X=\E\sum_{i=1}^n\V_i(x)W_i(x),
\]
and let $Y$ be the space of real-valued functions on $\Omega_n$, equipped with
$\ip{g}{h}_Y=\E[gh]$.  Let $A:X\to Y$ be the divergence map.  For every
$\V\in X$ and $\phi\in Y$, changing variables $x\mapsto x^{\oplus i}$ gives
\[
\begin{aligned}
   \ip{A\V}{\phi}_Y
   &=\sum_{i=1}^n\E[(\V_i(x)-\V_i(x^{\oplus i}))\phi(x)]\\
   &=\sum_{i=1}^n\E[\V_i(x)(\phi(x)-\phi(x^{\oplus i}))]\\
   &=2\E\sum_{i=1}^n\V_i(x)D_i\phi(x).
\end{aligned}
\]
Thus the adjoint of $A$ is
\begin{equation}\label{eq:divergence-adjoint}
   A^*\phi=2D\phi.
\end{equation}

Define
\[
   J(\V)=\E\|\V(x)\|_2.
\]
The primal problem is
\[
   \inf\{J(\V):A\V=Lf\}.
\]
The functional $J$ is finite and continuous on all of the finite-dimensional
space $X$, and the affine constraint is feasible; for example,
$\V_i=f/2$ satisfies $A\V=Lf$.  This is the standard continuity qualification in the finite-dimensional Fenchel--Rockafellar theorem~\cite[Theorem~31.2]{Rockafellar}, so there is no duality gap.

For completeness, introduce a multiplier $\phi\in Y$ and form the Lagrangian
\[
   J(\V)+\ip{\phi}{Lf-A\V}_Y.
\]
Using \eqref{eq:divergence-adjoint}, the infimum over $\V$ separates over vertices.  At a fixed $x$ it is, up to the common positive factor $2^{-n}$,
\[
   \inf_{v\in\R^n}\big\{\|v\|_2-2v\cdot D\phi(x)\big\}.
\]
By the duality between the Euclidean norm and itself, this infimum is $0$ if
$2\|D\phi(x)\|_2\le1$ and is $-\infty$ otherwise.  Hence the dual problem is
\[
   \sup\left\{
      \E[\phi Lf]:2\|D\phi(x)\|_2\le1
      \text{ for every }x\in\Omega_n
   \right\}.
\]
Strong duality identifies this supremum with the primal value $\B(f)$, proving
\eqref{eq:dual}.
\end{proof}

\end{document}